\documentclass[a4paper,12pt,reqno]{amsart}

\usepackage[T1]{fontenc}
\usepackage[utf8]{inputenc}
\usepackage[english]{babel}
\usepackage[left=25mm,right=25mm,top=35mm,bottom=35mm]{geometry}
\usepackage{amsmath,amssymb,amsthm}
\usepackage{pgfplots}
\usepackage{pgfplotstable}
\usepackage[margin=0pt,font={small}]{caption}
\usepackage{booktabs}
\usepackage{color,colortbl}
\usepackage[colorlinks=true,allcolors=blue,breaklinks]{hyperref}

\usepackage{algpseudocode} 
\usepackage{multirow} 
\usepackage{graphicx} 
\graphicspath{{./fig/}} 
\usepackage{caption} 
\usepackage{subcaption} 

\font\msbm=msbm10

\def\vfld#1{\vec{\,#1}}
\def\P{\hbox{\msbm P}}
\def\Q{\hbox{\msbm Q}}

\newcommand{\RE}{\mathcal R}

\newcommand\0{\boldsymbol{0}}

\newcommand{\bfu}{\mathbf{u}} 
\newcommand{\bfp}{\mathbf{p}}
\newcommand{\bfq}{\mathbf{q}}
\newcommand{\bfr}{\mathbf{r}}
\newcommand{\bff}{\mathbf{f}} 
\newcommand{\bfg}{\mathbf{g}} 
\newcommand{\bfz}{\mathbf{z}}

\renewcommand{\P}{\mathbb{P}}
\newcommand{\R}{\mathbb{R}}

\renewcommand{\AA}{\mathcal{A}}

\newcommand{\MM}{\mathcal{M}}

\newcommand{\PP}{\mathcal{P}}

\newcommand{\TT}{\mathcal{T}}

\DeclareMathOperator*{\hull}{span}

\newcommand{\npressure}{n_p}
\newcommand{\nvelocity}{n_u}
\newcommand{\nCpressure}{n_{k}}
\newcommand{\nel}{n_{el}}

\newcommand{\bp}{\boldsymbol{p}}
\newcommand{\bq}{\boldsymbol{q}}
\newcommand{\br}{\boldsymbol{r}}
\newcommand{\bu}{\boldsymbol{u}}
\newcommand{\bv}{\boldsymbol{v}}
\newcommand{\bw}{\boldsymbol{w}}
\newcommand{\bz}{\boldsymbol{z}}

\newcommand{\qnull}{\boldsymbol{k}}
\newcommand{\anull}{\boldsymbol{w}}
\newcommand{\zrange}{\overline{\bz}}

\newcommand{\boldA}{\boldsymbol{A}}
\newcommand{\pressuremass}{M_Q} 
\newcommand{\Aapprox}{\boldsymbol{M}}
\newcommand{\Qapprox}{M_S}

\newcommand{\Vspace}{\boldsymbol{V}_h}
\newcommand{\Qspace}{Q_h}
\newcommand{\Qzerospace}{Q_{h}^0}
\newcommand{\VspaceTH}{\boldsymbol{V}_{h}^{TH}}
\newcommand{\QspaceTH}{Q_h^{TH}}
\newcommand{\VspaceETH}{\VspaceTH}
\newcommand{\QspaceETH}{Q_{h}^{\star}}

\DeclareMathOperator*{\diag}{diag}
\DeclareMathOperator*{\nullspace}{null}

\def\marginnote#1{\setbox0=\vtop{\hsize 5pc
    \rightskip=.5pc plus 1.5pc #1}\leavevmode
     \vadjust{\dimen0=\dp0
      \kern-\ht0\hbox{\kern\hsize\kern1pc\box0}\kern-\dimen0}}
\def\bookref{}
\def\backref{}

\pgfplotsset{
tick label style={font=\tiny},
legend style={font=\tiny},
xlabel style={yshift=+0.5ex},
ylabel style={yshift=-1.0ex}
}

\theoremstyle{plain}
\newtheorem{theorem}{Theorem}
\newtheorem{proposition}[theorem]{Proposition}

\theoremstyle{definition}
\newtheorem{algorithm}[theorem]{Algorithm}
\newtheorem{remark}[theorem]{Remark}
\newtheorem{examp}{Test problem}

\makeatletter
\def\@seccntformat#1{%
  \protect\textup{\protect\@secnumfont
    \ifnum\pdfstrcmp{subsection}{#1}=0 \bfseries\fi
    \csname the#1\endcsname
    \protect\@secnumpunct
  }%
}
\makeatother

\usepackage{fancyhdr}
\advance\footskip0.5cm
\definecolor{otherblue}{rgb}{0,0.3,0.6}

\title{Fast solution of incompressible flow problems with two-level pressure approximation}
\author{Jennifer Pestana}
\address{Department of Mathematics and Statistics, University of Strathclyde, Glasgow, G1 1XH, UK}
\email{jennifer.pestana@strath.ac.uk}

\author{David J. Silvester}
\address{Department of Mathematics, University of Manchester, Oxford Road, Manchester M13 9PL, UK}
\email{d.silvester@manchester.ac.uk}

\thanks{{\em Acknowledgements.}
{This work was supported by EPSRC grant EP/W033801/1.}
}
\date{\today}

\begin{document}

\begin{abstract}
This paper develops efficient preconditioned iterative solvers for 
incompressible flow problems discretised by an enriched Taylor--Hood mixed approximation, 
in which the usual pressure space is augmented by a piecewise constant pressure to ensure 
local mass conservation. 
This enrichment process causes over-specification of the pressure when the pressure 
space is defined by the union of standard Taylor--Hood basis functions and piecewise constant 
pressure basis functions, which  
complicates the design and implementation of efficient solvers for the resulting linear systems. 
We first describe the impact of this choice of pressure space specification on the matrices involved. 
Next, we show how 
to recover effective solvers for Stokes problems, with preconditioners based on the singular pressure 
mass matrix, and for Oseen systems arising from linearised Navier--Stokes equations, by using 
a two-stage pressure convection--diffusion strategy.
The codes used to generate the numerical results are available online.
\end{abstract}

\maketitle
\thispagestyle{fancy}

\section{Introduction} \label{sec:intro}

Reliable and efficient iterative solvers for  models of steady incompressible flow emerged in the  
early 1990s.  Strategies based on block preconditioning of the underlying matrix operators 
using (algebraic or geometric) multigrid components have proved to be the key to  realising  mesh 
independent convergence (and optimal complexity) without the need for tuning parameters, particularly
in the context of classical  mixed finite element approximation, see Elman 
et al.~\cite[chap.\,9]{elman14}.  \bookref 
The focus  of this contribution is on efficient solver strategies in cases 
where (an inf--sup)  stable Taylor--Hood mixed approximation is augmented by a piecewise constant 
pressure in order to guarantee local conservation of mass. The augmentation leads to 
over-specification of the pressure solution when the pressure space is defined in the 
natural way via the union of Taylor--Hood pressure basis functions and basis functions for the piecewise 
constant pressure space, requiring a redesign of the established solver technology. 

The idea of adding a piecewise constant pressure to the standard rectangular biquadratic velocity, bilinear
pressure  ($\Q_2$--$\Q_1$) approximation was originally suggested during discussion around a
 blackboard at a conference on finite elements in fluids held in Banff in 1980; 
 see Gresho et al.~\cite{gresho81}. The need for local mass conservation was motivated by competition 
 from finite volume methods (such as the MAC scheme) in  the design of effective strategies for   
 modelling buoyancy-driven flow  in the atmosphere. 

The extension of the augmentation idea to Taylor--Hood ($\P_2$--$\P_1$)  triangular approximation was 
proposed in a paper presented at a conference  in Reading  in 1982;  see Griffiths~\cite{griffiths82}.
A  proof of stability of the augmented $\P_2$--$\P_1^\ast$  
approximation on triangular meshes  was constructed by Thatcher \& Silvester~\cite{thatcher87} in 1987. 
An extended version of this manuscript included discussion of $\Q_2$--$\Q_1^\ast$ hexahedral elements~\cite{thatcher90}. 
A rigorous assessment of the augmentation strategy was undertaken by Boffi et al.\ 
 two decades  later~\cite{boffi12}.
The strategy of augmenting a continuous pressure approximation to give a  locally mass-conserving 
strategy can also be generalised  to higher-order $\Q_k$--$\Q_{k-1}$ and  $\P_k$--$\P_{k-1}$ Taylor--Hood 
approximations. Inf--sup stability is assured for $k\geq 2$ in two dimensions and  for
 $\P_k$--$\P_{k-1}$ with $k\geq 3$ in three  dimensions; see~\cite[p.\,506]{boffi13}.
 
A closely related, yet fundamentally different Stokes approximation strategy is 
 developed in~\cite{yi2022}. Their idea is to stabilise the lowest-order  $\P_1$--$\P_0$  approximation  
 space  defined on triangles or tetrahedra  by augmenting  the  continuous velocity approximation 
 by a piecewise constant approximation. The extended mixed approximation is stable but nonconforming. 
 A similar approach is taken in~\cite{CGRT18}, but an alternative weak form is used. 
 Note that the design of linear solvers for the resulting discrete equations is relatively straightforward. 
 A special feature of the  extended approximation is that it can be augmented by postprocessing 
 to give a pressure-robust approximation; see~\cite{hu24}.

More generally, such an extended Galerkin  (EG) approximation  can be viewed as an intermediate between  continuous Galerkin (CG) and discontinuous Galerkin (DG).  This interpretation
of piecewise constant augmentation was originally put forward by  Sun \& Liu~\cite{sun09} and is
motivated by the fact that it gives local flux conservation when modelling transport in porous 
media flow problems, but with fewer degrees of  freedom compared to  vanilla DG.
 
 The novel contribution in this work  lies in the linear algebra aspects of two-field
pressure approximation. The immediate  issue that needs to be dealt with is the fact 
that, when the pressure space is specified by the frame formed by adding
usual finite element basis functions for the underlying Taylor--Hood space to basis 
functions for the discontinuous 
pressure space, the mass matrix  that determines the {\it stability} of the resulting mixed
approximation is  {\it singular}. This aspect is noted in the context of EG approximation
in~\cite[Remark 4.1]{lee16} and is discussed in  section~\ref{sec:mass}. 
The main issue that has  to be dealt with in practical flow simulation is the over-specification 
and associated ill-conditioning of the discrete operators that arise in the 
preconditioning of the linearised Navier--Stokes operator. This is the focus of the 
discussion in section~\ref{sec:oseen}. The conclusion of this
study is that  {\it optimal} complexity convergence rates can be recovered
when using two-field pressure  approximation, but only after  
a careful redesign of  the preconditioning components.

\section{Two-field pressure mass matrix} \label{sec:mass}

In this section we consider the Stokes problem 
\begin{alignat*}{2}
-\triangle \vfld{u} + \nabla p &= \vfld{0} & \quad & \text{ in } \Omega,\\
\nabla\cdot \vfld{u} & = 0& \quad & \text{ in } \Omega,\\
\vfld{u} &= \vfld{g} & \quad & \text{ on } \partial \Omega,
\end{alignat*}
where $\vfld{u}$ and $p$ are the fluid velocity and pressure, 
respectively, and $\Omega \subset \R^d$, 
$d \in \{2,3\}$, is a polygonal or polyhedral domain.
 
Throughout this section, we assume that
 $\Vspace \subset H_0^1(\Omega)^d$ and 
$\Qspace \subset L^2_0(\Omega):=\{q \in L_2 : \int_\Omega q = 0\}$
are an inf--sup stable pair of finite element spaces. 
Then the corresponding finite element 
approximation problem is to 
find $(\vfld{u}_h, p_h) \in \Vspace \times \Qspace$ such that 
\begin{subequations}
\label{eq:stokes_fe}
\begin{alignat}{2}
a(\vfld{u}_h,\vfld{v}_h) + b(\vfld{v}_h,p_h) 
&= (\vfld{f}_h,\vfld{v}_h) & \quad & \text{ for all } \vfld{v}_h \in \Vspace,\label{eq:stokes_fe_momentum}\\
b(\vfld{u}_h,q_h) & =  (g_h,q_h)
& \quad & \text{ for all } q_h \in \Qspace,\label{eq:stokes_fe_mass}
\end{alignat}
\end{subequations}
where $(\cdot,\cdot)$ denotes the usual $L^2$ inner product, 
right-hand side functions $\vfld{f}_h$ and $g_h$ are  associated with the 
specified boundary velocity  ($g_h$ is zero for enclosed flow),
and 
\begin{align*}
a(\vfld{u},\vfld{v}) = \int_\Omega \nabla \vfld{u} : \nabla \vfld{v},\qquad
b(\vfld{u},p)  = -\int_\Omega p\,\nabla\cdot \vfld{u}.
 \end{align*}
 
Solving the finite element problem \eqref{eq:stokes_fe} 
is then equivalent to solving the  linear system 
\begin{equation}
\label{eq:saddle_point}
\underbrace{\begin{bmatrix}
\boldA & B^T\\
B & 0
\end{bmatrix}}_{\AA}
\begin{bmatrix}
\bu\\
\bp
\end{bmatrix} 
=
\begin{bmatrix}
\boldsymbol{f}\\
\boldsymbol{g}
\end{bmatrix},
\end{equation}
where $\boldA\in\R^{\nvelocity \times \nvelocity}$ is symmetric positive definite 
and $B\in\R^{\npressure\times \nvelocity}$ \cite[chap.\,3]{elman14}.  \bookref
The matrix $\AA$ is of well-known saddle point type, 
and solvers for this sort of system have 
been extensively studied, see, e.g., Benzi et al.~\cite{BGL05}.
Since $\AA$ is large and sparse, the system is typically solved by an iterative method, 
with preconditioned MINRES \cite{PaSa75} a popular choice. 

An ideal preconditioner for $\AA$ is \cite{Kuzn95,MGW00}
\[\PP_{\text{ideal}} = \begin{bmatrix}\boldA & \\ & B\boldA^{-1}B^T\end{bmatrix},\]
since the eigenvalues of the preconditioned matrix are $0$, $1$, and $(1\pm \sqrt{5})/2$, 
with the zero eigenvalue appearing only if the preconditioned matrix is singular. 
This preconditioner is usually too costly to apply, and 
so efficient block diagonal preconditioners for $\AA$ are typically
 based on spectrally equivalent preconditioners for $\boldA$ and 
the negative Schur complement $S = B\boldA^{-1}B^T$. 
For Stokes problems, the solve with $\boldA$ can be replaced by, 
e.g., an algebraic or geometric multigrid solver, 
while for stable discretisations $S$ is spectrally equivalent to $\pressuremass$, 
the pressure mass matrix \cite[chap.\,3]{elman14}, \bookref
i.e., there exist constants $\gamma$ and $\Gamma$, independent of $h$, such that 
\begin{equation}
\label{eq:spec_q}
\gamma^2 \le \frac{\bq^TB\boldA^{-1}B^T\bq}{\bq^T\pressuremass\bq} \le \Gamma^2
\end{equation}
for all vectors $\bq$ except those corresponding to the function $q_h \equiv 1$ on $\Omega$. 
 For certain element pairs, $\pressuremass$ itself is easily inverted, e.g., 
for discontinuous $\P_0$ or $\Q_0$ pressures, the mass matrix 
is diagonal. 
Otherwise, $\pressuremass$ can be 
replaced by its diagonal or by a fixed number of steps of Chebyshev semi-iteration acceleration 
of a Jacobi iteration 
when, as is common, $\diag(\pressuremass)^{-1}\pressuremass$ has 
eigenvalues within a small interval 
and this interval lies away from the origin \cite{WaRe09,Wath87}.

To summarise, for inf--sup stable finite element pairs, 
an effective preconditioner for \eqref{eq:saddle_point} is 
\begin{equation}
\label{eq:stokes_pre}
\PP = \begin{bmatrix} \Aapprox & \\ & \Qapprox \end{bmatrix},
\end{equation}
where $\Aapprox \in \R^{\nvelocity \times \nvelocity}$ is $\boldA$ or an approximation, 
and $\Qapprox \in \R^{\npressure\times\npressure}$ 
is $\pressuremass$ or an approximation. 
(See, e.g., Elman et al.~\cite[chap.\,4]{elman14} \bookref 
for results with  $\Q_2$--$\Q_1$ approximation on quadrilaterals.)

In this section, our aim is to determine effective preconditioners for the enriched Taylor--Hood element.
As we will see, although enriching the pressure space results in better mass conservation properties, 
it can also present challenges for solving \eqref{eq:saddle_point}. 

\subsection{Augmented Taylor--Hood Elements}
We see from \eqref{eq:stokes_fe} that the mass conservation condition
 $\nabla\cdot \vfld{u} = 0$ is imposed only in a weak sense, 
 and if Taylor--Hood elements are employed 
 we can only guarantee that \eqref{eq:stokes_fe_mass} will hold.  
 However, by augmenting the pressure space by piecewise constant pressures 
 it is possible to obtain local mass conservation, 
 so that  the average of the divergence is zero on each individual element. 
The natural choice is to define this space by a frame 
 consisting of the union 
 of the continuous Taylor--Hood pressure basis functions and the discontinuous 
 pressure basis functions.  In this case, as we will discuss later in this section,
constant pressures have multiple 
 representations, 
 and this results in certain challenges for solving \eqref{eq:saddle_point}. 
Although it is certainly possible to instead compute a basis for this augmented 
 pressure space, we show that the linear algebra challenges associated with 
 using a frame can be overcome, which simplifies the implementation of the method.

Let us now describe the enriched Taylor--Hood finite element spaces. 
We first introduce a shape-regular family of simplicial, 
quadrilateral (in 2D) or hexahedral (in 3D) decompositions of the domain $\Omega$. 
We assume that any two elements have at most a common face, 
edge, or vertex and denote  by $h$ the maximum diameter of any element. 
The total number of elements in the resulting mesh is $\nel$. 

We denote the usual Taylor--Hood finite element space by 
$\VspaceTH \times \QspaceTH$, so that $\VspaceTH\times \QspaceTH = (\Q_{k+1})^d \times \Q_{k}$ or 
 $\VspaceTH\times \QspaceTH = (\P_{k+1})^d \times \P_{k}$ with $d\geq 2$. 
In the latter case, we additionally assume that the polynomial degree, $k$, satisfies $k \ge d-1$. 
The corresponding enriched Taylor--Hood space is 
$\VspaceETH \times \QspaceETH$ where 
\begin{align}
\QspaceETH &= \{q = q_k + q_0, q_k \in \QspaceTH, q_0 \in \Qzerospace\}\label{eq:eth_pressure_space}
\end{align}
and 
$\Qzerospace$ is the space of discontinuous pressures that are constant on each element.  
Thus, we see that the velocity approximation space is identical to that of the corresponding 
Taylor--Hood element, 
while $\QspaceETH$ is $\QspaceTH$ augmented with piecewise constant pressures. 
It follows that functions in $\QspaceETH$ may be discontinuous across inter-element boundaries. 

We stress that the enriched Taylor--Hood space is well defined, and inf--sup stable.
From a linear algebra perspective, however, a critical point is the representation of $\QspaceETH$. 
Ideally, we would like to represent functions in $\QspaceETH$ as linear combinations of basis functions.
In this case the resulting linear system would be nonsingular, and it is likely that the 
approaches described at the stat of this section could be applied directly.
However, the most natural choice is define $\QspaceETH$ by a frame, i.e., to let 
\begin{equation}
\label{eq:eth_spec}
\QspaceETH = \hull\{\phi_1,\dotsc, \phi_{\nCpressure},\phi_{\nCpressure+1},\dotsc,\phi_{\npressure}\},
\end{equation}
where $\{\phi_k\}_{k = 1}^{\nCpressure}$ and $\{\phi_k\}_{k = \nCpressure+1}^{\npressure}$
are Lagrange bases of $\QspaceTH$ and $\Qzerospace$ (see \eqref{eq:eth_pressure_space}). 
This approach to specifying the pressure space has been used in, e.g., \cite{boffi12, QiZh05,thatcher90}.

Although this makes specification of the pressure straightforward, with this choice
any constant function on $\Omega$ 
can be  represented by a function in $\QspaceTH$ \emph{or} a function in $\Qzerospace$. 
This has profound consequences for the linear algebra: the
 pressure mass matrix $\pressuremass$ that determines $\Qapprox$ in \eqref{eq:stokes_pre} becomes singular, 
and the rank of the matrix $B^T$ is reduced. 
In the rest of this section we first establish these properties, before showing that 
it is still possible to solve~\eqref{eq:saddle_point} by preconditioned MINRES in this case.

Using \eqref{eq:eth_spec}, we can relate vectors $\bp \in \R^{\npressure}$ 
and functions $p=p_k+p_0\in \QspaceETH$, where $p_k \in \QspaceTH$, $p_0 \in \Qzerospace$. 
Specifically, 
\begin{equation}
\label{eq:function_basis}
p = \sum_{i = 1}^{\npressure} \bp_i \phi_i, \quad 
p_k = \sum_{i = 1}^{\nCpressure} \bp_i \phi_i \quad \text{ and }\quad  
p_0 = \sum_{i = \nCpressure+1}^{\npressure} \bp_i \phi_i.
\end{equation}

As mentioned above, any constant function has multiple representations when specified via \eqref{eq:eth_spec}. 
In particular, $p = p_k + p_0 \equiv 0$ with  $p_k = \alpha$ and $p_0 = -\alpha$, for any $\alpha \in \R$.
 We see from \eqref{eq:function_basis} that this representation of the zero function corresponds to the 
 vector $\alpha \qnull$, where 
\begin{equation}
\label{eq:pressure_null_vector}
\qnull = 
\begin{bmatrix}
\boldsymbol{1}_{n_k}\\
-\boldsymbol{1}_{n_0}
\end{bmatrix}
\end{equation}
and $n_0 = n_p-n_k$. 
A direct consequence of the correspondence between $\qnull$ and the zero function on $\Omega$
 is that $\pressuremass\qnull = \boldsymbol{0}$ and $B^T\qnull = \boldsymbol{0}$, 
 as we now show. 

\begin{proposition}
\label{prop:Qnull}
Let 
\[\pressuremass = [q_{ij}]_{i,j = 1}^{\npressure}, \quad q_{ij} = \int_\Omega \phi_i \phi_j,\] 
be the pressure mass matrix for the enriched Taylor--Hood pressure space 
$\QspaceETH$ in \eqref{eq:eth_pressure_space}, specified by the frame \eqref{eq:eth_spec}. 
Then, $\nullspace(\pressuremass) = \hull\{\qnull\}$, where $\qnull$ is given in \eqref{eq:pressure_null_vector}. 
\end{proposition}

\begin{proof}
Let $\bp \in \R^{\npressure}$, $\bp\ne \boldsymbol{0}$. Then, using \eqref{eq:function_basis}, we find that  
\[
\pressuremass\bp = 0 \quad 
 \Leftrightarrow \quad \bp^T \pressuremass \bp=0 \quad
 \Leftrightarrow \quad  \int_\Omega\left(\sum_{i = 1}^{\npressure} \bp_i \phi_i\right) 
\left( \sum_{j = 1}^{\npressure}\bp_j \phi_j\right)  = 0
 \Leftrightarrow \quad  \int_\Omega p^2  = 0,
\]
which implies that $p \equiv 0$ in $\Omega$
since $p$ is continuous on each element. 
Since $p \equiv 0$ corresponds to vectors of the form $\alpha \qnull$,
we find that $\hull\{\qnull\} \subseteq \nullspace(\pressuremass)$. 

We now show that there are no other vectors in the nullspace. 
Since $p = p_k + p_0\equiv 0$, $p_k \in \QspaceTH$, $p_0 \in \Qzerospace$, 
we must have $p_k = -p_0$ everywhere in $\Omega$. But $p_0$ is piecewise constant, 
and $p_k$ is continuous on $\overline{\Omega}$. It follows that $p_0\equiv \alpha$ and 
$p_k \equiv -\alpha$ on $\Omega$ for some constant $\alpha\in\R$.
Such functions correspond to vectors of the form $\alpha \qnull$, 
which shows that  $\nullspace(\pressuremass) = \hull\{\qnull\}$, as required.
\end{proof}

\begin{remark}
\label{rem:Bnull}
A very similar argument shows that $\qnull \in \nullspace(B^T)$. 
\end{remark}

What do 
these results mean for the solution of \eqref{eq:saddle_point}
 by preconditioned MINRES, when the coefficient matrix is constructed using the frame in \eqref{eq:eth_spec}? 
 First, 
if $B^T\qnull = \boldsymbol{0}$  then it
 follows that $\AA$ is always singular, with $\AA\anull = \boldsymbol{0}$, 
 where 
\begin{equation}
\label{eq:anull}
\anull = \begin{bmatrix} \boldsymbol{0}_{\nvelocity}\\ \qnull \end{bmatrix}.
\end{equation}
If the linear system \eqref{eq:saddle_point} is consistent then this 
does not pose a problem for preconditioned MINRES.
However, as we shall now see, the proposed block diagonal preconditioner  \eqref{eq:stokes_pre}, 
with $\Qapprox = \pressuremass$,  
is also singular, and so effective preconditioning requires some care.

\begin{algorithm}{Preconditioned MINRES algorithm for solving $\AA \boldsymbol{x} = \boldsymbol{b}$ 
with symmetric positive definite preconditioner $\PP$ \cite[Algorithm 4.1]{elman14}.  \bookref}
\label{alg:minres}
\begin{algorithmic}[1]
\State $\bv^{(0)} = \boldsymbol{0}$, $\bw^{(0)} =  \boldsymbol{0}$, $\bw^{(1)} =  \boldsymbol{0}$, $\gamma_0 = 0$
\State Choose $\boldsymbol{x}^{(0)}$, compute $\bv^{(1)} = \boldsymbol{b} - \AA\boldsymbol{x}^{(0)}$
\State Solve $\PP \bz^{(1)} = \bv^{(1)}$, set $\gamma_1 = \sqrt{\langle \bz^{(1)}, \bv^{(1)}\rangle}$
\State Set $\eta = \gamma_1$, $s_0 = s_1 = 0$, $c_0 = c_1 = 1$
\For {$j=1$ until convergence}
\State $\bz^{(j)} = \bz^{(j)}/\gamma_j$
\State $\delta_j = \langle A\bz^{(j)}, \bz^{(j)}\rangle$
\State $\bv^{(j+1)} = \AA\bz^{(j)} - (\delta_j/\gamma_j)\bv^{(j)}-(\gamma_j/\gamma_{j-1})\bv^{(j-1)}$
\State Solve $\PP \bz^{(j+1)} = \bv^{(j+1)}$
\State $\gamma_{j+1} = \sqrt{\langle \bz^{(j+1)}, \bv^{(j+1)}\rangle}$
\State $\alpha_0 = c_j\delta_j - c_{j-1}s_j\gamma_j$
\State $\alpha_1 = \sqrt{\alpha_0^2 + \gamma_{j+1}^2}$
\State $\alpha_2  = s_j\delta_j + c_{j-1}c_j\gamma_j$
\State $\alpha_3 = s_{j-1}\gamma_j$
\State $c_{j+1} = \alpha_0/\alpha_1$; $s_{j+1} = \gamma_{j+1}/\alpha_1$
\State $\bw^{(j+1)} = (\bz^{(j)} - \alpha_3\bw^{(j-1)} - \alpha_2\bw^{(j)})/\alpha_1$
\State $\boldsymbol{x}^{(j)} = \boldsymbol{x}^{(j-1)} + c_{j+1}\eta \bw^{(j+1)}$
\State $\eta = -s_{j+1}\eta$
\State <Test for convergence>
\EndFor
\end{algorithmic}
\end{algorithm}

\subsection{Preconditioning considerations}
\label{sec:singular_preconditioner}
Knowing that the enriched Taylor--Hood element is inf--sup stable
implies that the preconditioner \eqref{eq:stokes_pre} will be effective when solving \eqref{eq:saddle_point}. 
However, with the common choice of frame \eqref{eq:eth_spec}, the matrix $\pressuremass$, and hence the preconditioner $\PP$, are singular, 
since $\PP\anull = \boldsymbol{0}$, where $\anull$ is given in \eqref{eq:anull}. 
(Note that this implies that  $\AA$ and $\PP$ have a common nullspace.) 
Here, we will show that using a singular preconditioner causes no difficulty for the preconditioned MINRES 
method in Algorithm~\ref{alg:minres} in exact arithmetic. At the end of section~\ref{sec:mass}, 
we present numerical results with two different preconditioners to demonstrate our approach.

We start by noting that at each  iteration step we need to 
solve a linear system of the form $\PP\bz^{(j)} = \bv^{(j)}$. 
If this system is consistent there are infinitely many solutions, 
which take the form $\bz^{(j)} = \zrange^{(j)} + \zeta_j \anull$, 
with $\zrange^{(j)}\perp \anull$ and $\zeta_j \in \R$. 
Our next step is to examine the effect of $|\zeta_j|$ on the scalars and vectors in 
Algorithm~\ref{alg:minres}. 
Observe that if the linear system $\AA \boldsymbol{x} = \boldsymbol{b}$ is 
consistent then $\boldsymbol{b}\perp \anull$
and so $\bv^{(1)} \perp \anull$. 
It then follows by induction that 
$\bv^{(j)}\perp \anull$, $j = 1,2,\dotsc$, and that the systems $\PP\bz^{(j)} = \bv^{(j)}$
are all consistent. 
Furthermore, 
$\gamma_j = \langle \bz^{(j)},\bv^{(j)}\rangle^{\frac{1}{2}} = \langle \zrange^{(j)},\bv^{(j)}\rangle^{\frac{1}{2}}$ because 
of the orthogonality of $\bv^{(j)}$ and $\anull$, so that $\zeta_j$ does not affect $\gamma_j$. 
Similarly, $\delta_j = \langle \AA \bz^{(j)},\bz^{(j)}\rangle =  \langle \AA \zrange^{(j)},\zrange^{(j)}\rangle$, 
which shows that $\delta_j$ is similarly unaffected by $\zeta_j$. 
Indeed, the only quantities that are affected by the nullspace components $\zeta_j \anull$, $j = 1,2,\dotsc$,  
are the vectors $\boldsymbol{w}^{(j)}$ and  $\boldsymbol{x}^{(j)}$.
In exact arithmetic, this is not a problem, since solutions of 
 $\AA \boldsymbol{x} = \boldsymbol{b}$ may certainly contain a component in the direction of $\anull$. 
However, in unlucky cases, the size of the nullspace component of $\boldsymbol{x}^{(j)}$ 
may be so large as to dominate the approximate solution. 
Alternatively, in finite precision $\bv^{(j)}$ and $\bw$ may not be exactly orthogonal. 
Hence, it may be wise to explicitly ensure that $\bz \perp \anull$. 
One option is to orthogonalise $\bz$ against $\anull$ after each preconditioner solve, 
but if $|\zeta_j|$ is large then the result may be inaccurate. 
A more robust approach is to note that solutions of $\PP\bz = \br$ are minimisers of the quadratic form 
\[\frac{1}{2}\bz^T\PP\bz - \bz^T\br,\]
since $\PP$ is positive semidefinite. 
Constraining $\bz$ to be orthogonal to $\anull$ is then equivalent to the following optimisation problem: 
\[\min_{\bz} \frac{1}{2}\bz^T\PP\bz - \bz^T\br \quad \text{s.t.} \quad \anull^T \bz = 0.\]
Applying a Lagrange multiplier approach results in the augmented system 
\[\begin{bmatrix} \PP & \anull\\ \anull^T & 0 \end{bmatrix} 
\begin{bmatrix} \bz\\\lambda\end{bmatrix} 
= 
\begin{bmatrix}\br\\0\end{bmatrix},\]
where $\lambda$ is the Lagrange multiplier, 
and solving this augmented system gives a vector $\bz$ that is orthogonal to $\anull$. 
Moreover, since 
\[
\begin{bmatrix} \PP & \anull\\ \anull^T & 0 \end{bmatrix}
 = 
 \begin{bmatrix} \Aapprox & 0& 0\\0 & \Qapprox & \qnull\\ 0& \qnull^T &0 \end{bmatrix}
 \]
we see that only the solve with $\Qapprox$ needs to be modified. 

\subsection{Approximating the two-level pressure mass matrix}
Now let us consider approximations of the matrix $\pressuremass$ which, 
because of the two-level pressure approximation, has $2\times 2$ block structure: 
\[\pressuremass = \begin{bmatrix} Q_{k} & R^T\\ R & Q_{0}\end{bmatrix},\]
where 
\begin{alignat*}{2}
Q_k &= [q_{k,ij}], \, i,j = 1,\dotsc, \nCpressure, &\qquad  q_{k,ij}& = \int_\Omega \phi_j \, \phi_i,\\
 R &= [r_{ij}], \, i = \nCpressure + 1,\dotsc, \npressure, \, j = 1,\dotsc, \nCpressure, & \qquad  r_{ij} &= \int_\Omega \phi_j \, \phi_i,\\
Q_0 &= [q_{0,ij}], \, i,j = \nCpressure + 1,\dotsc, \npressure, &\qquad  q_{0,ij}& = \int_\Omega \phi_j \, \phi_i.
\end{alignat*}
Note that $Q_k$ is the standard Taylor--Hood pressure mass matrix, 
and $Q_0$ is the standard discontinuous pressure mass matrix, 
while $R$ represents cross terms between the spaces 
$\QspaceTH$ and $\Qzerospace$ (see \eqref{eq:eth_pressure_space}).

Naively, one may wish to approximate $Q_k$ by its diagonal, but this results in slow convergence rates; 
it can be shown, using the method described by Wathen \cite{Wath87}, that $\diag(Q_k)^{-1}Q_k$ has 
eigenvalues in an interval $[0,\mu]$ where, for example, for $\P_1$ elements on triangles $\mu=3$. 
Accordingly, applying a fixed number of iterations of a Chebyshev semi-iteration based on a Jacobi iteration, as is standard 
for nonsingular Galerkin mass matrices \cite{WaRe09}, leads to large iteration counts here. 
Replacing $Q_k$ or $Q_0$ in $\pressuremass$ by their diagonals results in similarly poor performance.

However, it is possible to replace $\pressuremass$ by an approximation designed for symmetric 
positive semidefinite matrices, see, e.g., \cite{Cao91,Dax90,LNT16} and the 
references therein. 
For example, symmetric Gauss--Seidel is convergent for such matrices~\cite{Dax90}. 
Here, we have applied a fixed number of iterations of Chebyshev semi-iteration based on symmetric Gauss--Seidel~\cite{Cao91}.
The results are still mesh-dependent, because the largest non-unit eigenvalue of the underlying symmetric Gauss--Seidel 
iteration matrix approaches 1 as the mesh is refined but, for large problems, the cost per iteration is much lower than applying $\pressuremass$ 
exactly. 

\subsection{Reliable computation of the discrete inf--sup constant}
The inf--sup constant depends on the shape of the domain $\Omega$. 
Thus the constant for a step domain with a long channel is much smaller
than the constant for a square domain. Estimation of the inf--sup
constant dynamically provides useful information regarding the
connection between the velocity error and the pressure error. 
The inf--sup constant  also features in the norm equivalence between
the residual norm of the saddle-point system and the natural
``energy'' norm
$\| \nabla u \|_{L^2(\Omega)} + \| p \|_{L^2_0(\Omega)}, $
as discussed in the motivating paper \cite{SiSi11}.

A typical strategy to estimate the inf--sup constant for \eqref{eq:stokes_fe}
is to find the largest $\gamma$ that satisfies \eqref{eq:spec_q}, 
i.e., to find the smallest nonzero eigenvalue of the generalised eigenvalue problem 
\begin{equation}
\label{eq:inf-sup_eig}
B\boldA^{-1}B^T \bv = \lambda \pressuremass\bv,
\end{equation}
where $\boldA$ is prescribed by the velocity basis functions, $\pressuremass$ by the pressure functions, 
 and $B$ by both the velocity and pressure functions. 
 
Although the inf--sup constant is independent of the choice of these functions, in practice, its 
computation may be affected.  
For example, with the common choice to define $\QspaceETH$ by the frame \eqref{eq:eth_spec}, 
Proposition \ref{prop:Qnull} and Remark \ref{rem:Bnull}
show that $\qnull$ lies in the nullspaces of both $B^T$ and $\pressuremass$, 
which means that this generalised eigenvalue problem is singular, 
i.e., \emph{any} $\lambda\in\mathbb{R}$ satisfies \eqref{eq:inf-sup_eig} when $\bv = \qnull$. 
It is known that generalised eigenvalue problems with singular pencils are challenging to solve
numerically~\cite{HMP19}, 
and additional checks must be performed to ensure that an estimate of $\gamma$
 is not associated with the eigenvector $\qnull$. 
In practice we find that standard methods for computing eigenvalues of sparse 
 matrices may struggle to accurately compute these eigenvalues, 
 precisely because $B\boldA^{-1}B^T$ and $\pressuremass$ are singular. 

An alternative is to define $\QspaceETH$ by a basis instead of the commonly-used frame 
 \eqref{eq:eth_spec}, but it turns out that this is not necessary. 
The EST-MINRES approach proposed in \cite{SiSi11} 
is much more robust for this singular eigenvalue problem and gives consistently reliable results 
for certain preconditioners. 
The intuition is that by ensuring that any solves with $\pressuremass$ 
are orthogonal to $\qnull$ within the preconditioned MINRES method, 
(see Section~\ref{sec:singular_preconditioner}), we 
instead solve \eqref{eq:inf-sup_eig} for $\bv \perp \qnull$.  

We illustrate the EST-MINRES inf--sup constant estimates using two representative test problems 
that we describe below. 
All numerical results were obtained using T-IFISS~\cite{BRS21} (for triangular elements) and 
IFISS3D~\cite{PPS23} (for cubic elements). 
The stopping criterion for preconditioned MINRES 
is a reduction of the norm of the preconditioned 
residual by eight orders of magnitude,
 i.e., $\|\boldsymbol{r}_k\|_{\PP^{-1}}/\|\boldsymbol{r}_0\|_{\PP^{-1}} < 10^{-8}$. 
We apply two different block diagonal preconditioners \eqref{eq:stokes_pre}. The first is 
\[\mathcal{P}_1 = \begin{bmatrix} \boldA & \\ & \pressuremass\end{bmatrix}.\]
Although this preconditioner is expensive to apply, because it involves exact solves with 
$\boldA$, we have used it here to more clearly illustrate the key findings of this section, namely, 
that \eqref{eq:stokes_pre} with $\Qapprox =\pressuremass$ is an effective preconditioner for 
augmented Taylor--Hood problems, despite the singularity of $\pressuremass$, 
and that EST-MINRES provides reliable approximations to 
$\gamma$, the discrete inf--sup constant in \eqref{eq:spec_q}.
However, we note that it is possible to replace $\boldA$ by, e.g., an algebraic multigrid method 
such as the HSL code MI20 \cite{hsl}. With this AMG approximation, 
we are able to solve a 3D problem with nearly $10^6$ 
degrees of freedom in under 10 seconds 
on a quad-core, 62 GB RAM, Intel
i7-6700 CPU with 3.20GHz. Moreover, the effect on the inf--sup constant 
is fairly small.

Our second, cheaper, preconditioner is 
\[\mathcal{P}_2 = \begin{bmatrix} \boldA_{\text{AMG}} & \\ & M_{\text{cheb}}\end{bmatrix},\]
where $\boldA_{\text{AMG}}$ is one AMG V-cycle, with the default parameters in T-IFISS and 
IFISS3D and  $M_{\text{cheb}}$ is 20 iterations of 
a Chebyshev semi-iteration method based on symmetric Gauss--Seidel.

\begin{examp}[two-dimensional enclosed flow]\label{ex:2d_cavity}
Our first example is a classical driven-cavity flow in the square domain  $D=[-1,1]^2$. 
A Dirichlet no-flow condition is imposed on the 
bottom and side boundaries, while on the lid the nonzero tangential velocity is 
$u_y= 1- x^4$. 
The domain is subdivided uniformly into $n^2$ bisected squares. 
We use the standard $\P_2$--$\P_1$ Taylor--Hood 
mixed approximation and the augmented $\P_2$--$\P_1^\ast$ approximation. 
The two components of the pressure solution for the $\P_2$--$\P_1^\ast$ approximation, 
computed for $n=32$, are illustrated
in Fig.~\ref{fig:stokes_testproblem_2d}.   The centroid pressure field is  concentrated in the 
two corners where the pressure is singular, and the centroid pressures are  an  order of magnitude
smaller than  the vertex pressure in all elements. 
Consequently, the overall pressure field is visually identical to
 the $\P_1$ pressure field shown in the left plot in Fig.~\ref{fig:stokes_testproblem_2d}.
 \end{examp}
 
 \begin{figure}
       \begin{center}
	\includegraphics[width=0.7\linewidth,clip=true,trim = 2cm 4cm 1cm 2cm]{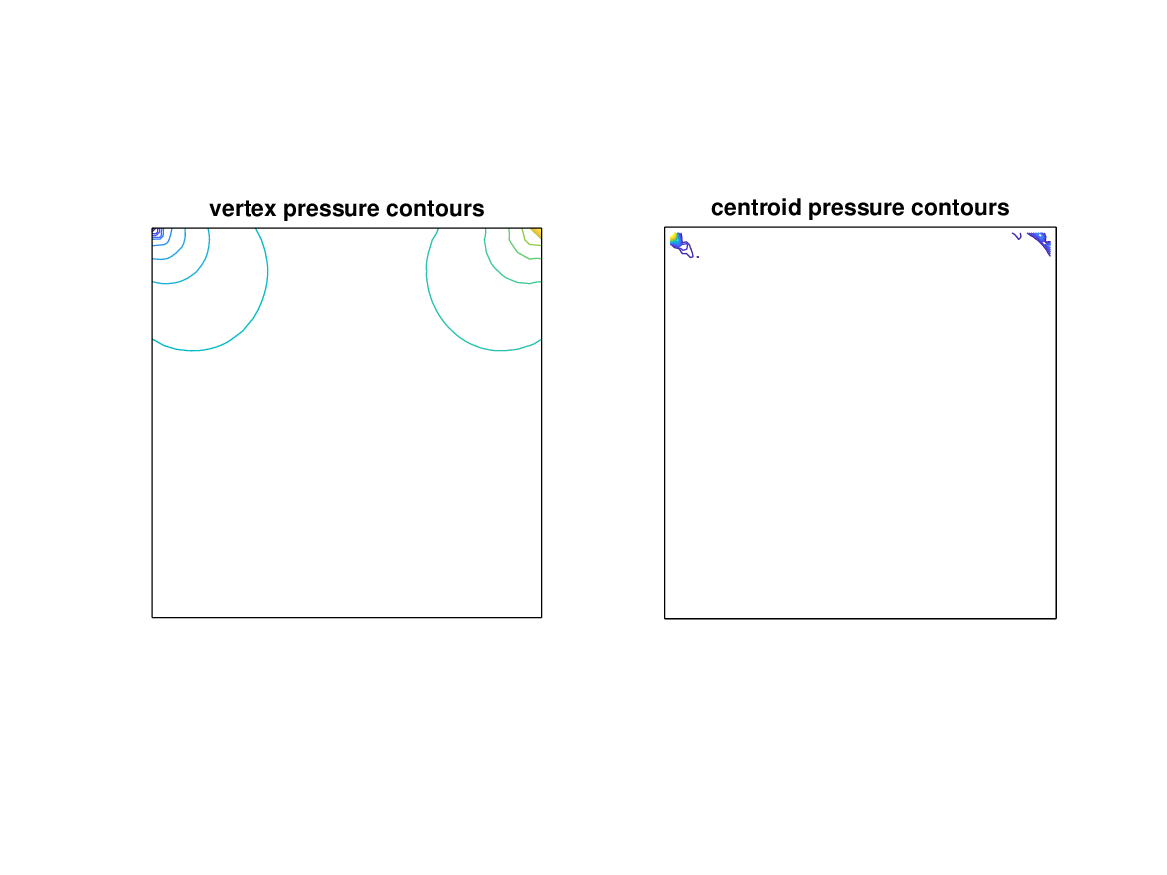}		 
	\end{center}
\caption{Representative $\P_2$--$\P_1^*$ pressure field solution for the cavity flow in 
Example \ref{ex:2d_cavity} computed on a uniform  mesh with 2048 right-angled  triangles and 1089 vertices.
Contours are equally spaced between the maximum
and minimum values in both plots.}
\label{fig:stokes_testproblem_2d}
\end{figure}

 \begin{examp}[three-dimensional enclosed flow]\label{ex:3d_cavity}
 Our second problem is a three-dimensional version of driven-cavity flow. 
 The domain is now $D = [-1,1]^3$. 
As in the previous example, the flow is enclosed, but now 
the nonzero tangential velocity 
$u_y= (1- x^4)(1-z^4)$ is specified on the top of the cavity.
The domain is subdivided uniformly into $n^3$ cubic elements, and 
we use $\Q_2$--$\Q_1$ and $\Q_2$--$\Q_1^\ast$ approximations. 
 \end{examp}

Table \ref{tab:2d_infsup} shows preconditioned MINRES iteration counts and EST-MINRES discrete inf--sup constant 
approximations for Example \ref{ex:2d_cavity} with preconditioner $\mathcal{P}_1$.
 We first note that for both the $\P_2$--$\P_1$ and $\P_2$--$\P_1^\ast$
approximations the iteration counts are quite similar and are mesh independent. 
In both cases the inf--sup constant approximations also appear to be converging from above. 
The approximation for $\P_2$--$\P_1^\ast$ elements appears to rapidly converge to four digits, 
indicating that even a relatively coarse grid is sufficient to obtain an approximation to the 
discrete inf--sup constant. 
However, the approximation for $\P_2$--$\P_1$ elements 
appears to converge more slowly. 
We also note that the two approaches give different inf--sup constant 
estimates, at least for the grids shown here.  
This is not so surprising as the matrices in \eqref{eq:spec_q} 
depend on the choice of finite element spaces. 

When we replace the preconditioner by the cheaper $\mathcal{P}_2$, 
the situation is largely unchanged 
for $\P_2$--$\P_1$ elements (see Table~\ref{tab:2d_infsup_sgscheb}). 
However, for $\P_2$--$\P_1^\ast$ elements, the preconditioned MINRES convergence rate 
and inf--sup constant estimate degrade as the mesh is refined. The effect on the convergence rate is 
caused by the largest non-unit eigenvalue of the symmetric Gauss--Seidel iteration matrix approaching 1 as the 
mesh is refined; this appears to also impact inf--sup constant approximation.

Fig.~\ref{fig:2d_infsup} plots the inf--sup approximations at each iteration of preconditioned MINRES 
for Example \ref{ex:2d_cavity} for preconditioner $\mathcal{P}_1$.
We see that a good approximation of the inf--sup constant is obtained 
after 20--25 iterations. It is again clear that for the enriched Taylor--Hood 
approximations we obtain very similar approximations for all grids. 

\begin{table}
\begin{tabular}{r| r r r r | r r r r}
\hline
 & \multicolumn{4}{c|}{$\P_2$--$\P_1$} & \multicolumn{4}{c}{$\P_2$--$\P_1^\ast$}\\
Grid& Velocity dof & Pressure dof & Iters & $\gamma^2$ & Velocity dof & Pressure dof & Iters & $\gamma^2$\\
\hline
4 & 2178 & 289 & 37 & 0.1947 & 2178 & 801 & 42 & 0.1397\\
5 & 8450 & 1089 & 37 & 0.1926 & 8450 & 3137 & 42 & 0.1396\\
6 & 33282 & 4225 & 39 & 0.1911 & 33282 & 12417 & 40 & 0.1395\\
7 & 132098 & 16641 & 37 & 0.1898 & 132098 & 49409 & 40 & 0.1395\\
8 & 526338 & 66049 & 37 & 0.1888 & 526338 & 197121 & 40 & 0.1395\\
\hline
\end{tabular}
\caption{Preconditioned MINRES iterations and discrete inf--sup constant approximations for Example \ref{ex:2d_cavity} 
and preconditioner $\mathcal{P}_1$}.
\label{tab:2d_infsup}
\end{table}

\begin{table}
\begin{tabular}{r| r r r r | r r r r}
\hline
& \multicolumn{4}{c|}{$\P_2$--$\P_1$} & \multicolumn{4}{c}{$\P_2$--$\P_1^\ast$}\\
Grid & Velocity dof & Pressure dof & Iters & $\gamma^2$ & Velocity dof & Pressure dof & Iters & $\gamma^2$\\
\hline
4 & 2178 & 289 & 42 & 0.1961 & 2178 & 801 & 59 & 0.1404\\
5 & 8450 & 1089 & 42 & 0.1940 & 8450 & 3137 & 59 & 0.1443\\
6 & 33282 & 4225 & 44 & 0.1925 & 33282 & 12417 & 117 & 0.0467\\
7 & 132098 & 16641 & 45 & 0.1917 & 132098 & 49409 & 152 & 0.0259\\
8 & 526338 & 66049 & 45 & 0.1912 & 526338 & 197121 & 197 & 0.0064\\
\hline
\end{tabular}
\caption{Preconditioned MINRES iterations and discrete inf--sup constant approximations for Example \ref{ex:2d_cavity} 
and preconditioner $\mathcal{P}_2$.}
\label{tab:2d_infsup_sgscheb}
\end{table}

\begin{figure}
\begin{center}
\begin{subfigure}[b]{0.4\textwidth}
         \centering
         \includegraphics[width=\textwidth]{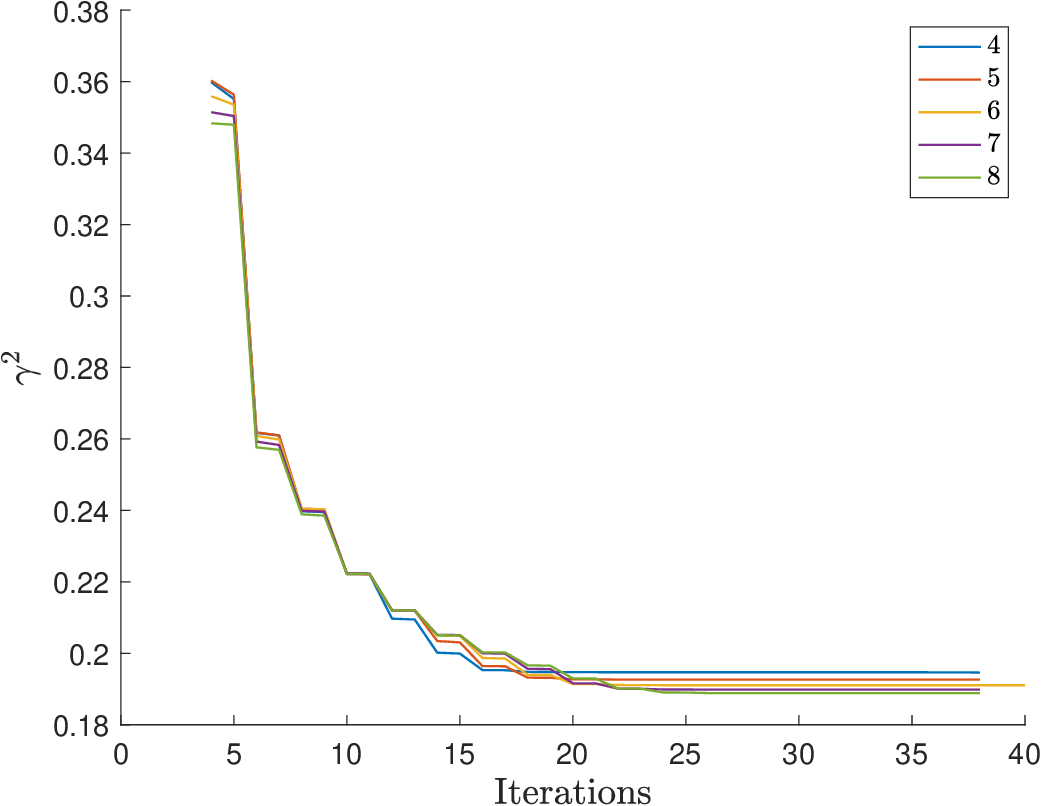}
     \end{subfigure}
     \begin{subfigure}[b]{0.4\textwidth}
         \centering
         \includegraphics[width=\textwidth]{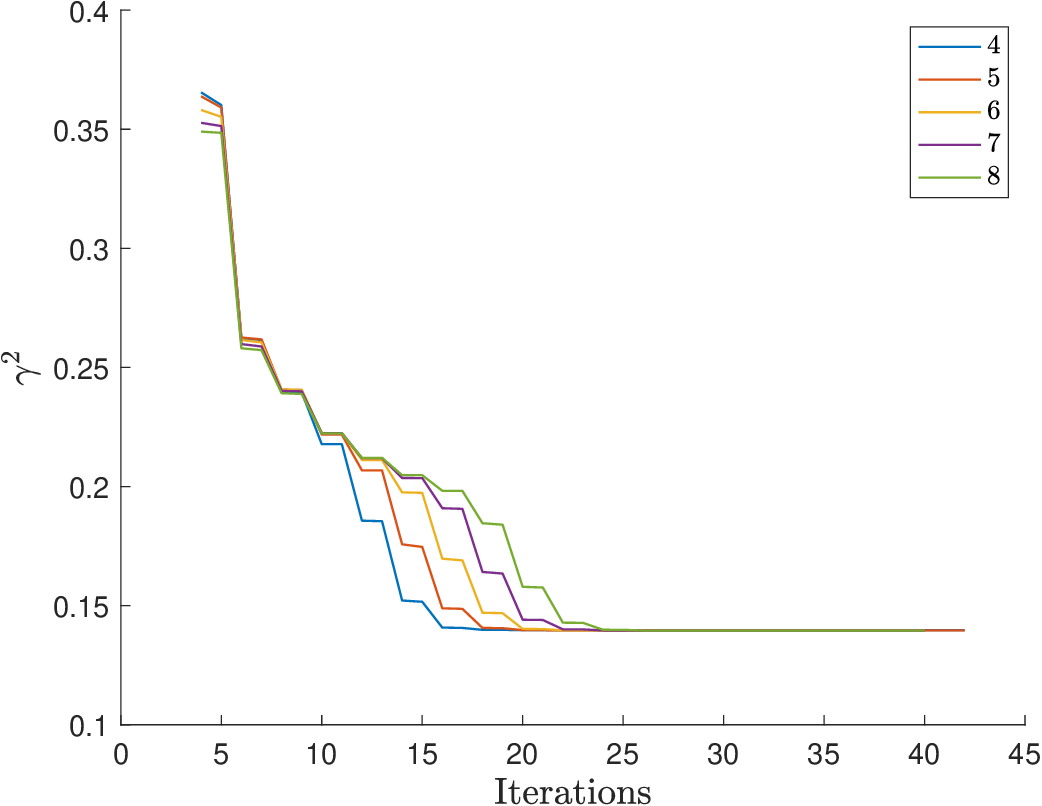}
     \end{subfigure}
     \end{center}
\caption{EST-MINRES estimates of the discrete inf--sup constant $\gamma^2$ at each iteration for Example \ref{ex:2d_cavity} 
and preconditioner $\mathcal{P}_1$ 
with $\P_2$--$\P_1$ (left) and $\P_2$--$\P_1^\ast$ (right) elements for the grids specified in Table \ref{tab:2d_infsup}.}
\label{fig:2d_infsup}
\end{figure}

The results for Example \ref{ex:3d_cavity} are given in Tables \ref{tab:3d_infsup} and 
\ref{tab:3d_infsup_sgscheb} and 
Fig.~\ref{fig:3d_infsup}.  Preconditioned MINRES iteration counts for $\mathcal{P}_1$ are again mesh independent and, for 
all but the coarsest mesh, are almost identical for the two element pairs. 
The discrete inf--sup constant approximations again appear to converge from above. Now, at least for the 
grids presented here, the estimates of 
$\gamma^2$  
for Taylor--Hood elements seem to ``track'' the  augmented Taylor--Hood estimates, i.e., the 
Taylor--Hood approximation on grid $j$ is almost the same as the augmented Taylor--Hood estimate on 
grid $j-1$. 
Table \ref{tab:3d_infsup_sgscheb} shows that using the cheaper $\mathcal{P}_2$ also results in mesh-independent convergence for 
$\Q_2$--$\Q_1$ elements.
Additionally, the inf--sup constant approximations seem to ``track'' the Taylor--Hood estimates in Table~\ref{tab:3d_infsup}. 
Now, however, there is more modest growth in the iteration counts when this approximate preconditioner is 
used for $\Q_2$--$\Q_1^\ast$ elements and the inf--sup approximations are much closer to those in Table~\ref{tab:3d_infsup}.

We also see from Fig.~\ref{fig:2d_infsup}, which plots the inf--sup approximations at each iteration of preconditioned MINRES, 
for preconditioner $\mathcal{P}_1$,
that the discrete inf--sup constant is already reasonably well approximated after 25 iterations.

\begin{table}
\begin{tabular}{r |r r r r | r r r r}
\hline
& \multicolumn{4}{c|}{$\Q_2$--$\Q_1$} & \multicolumn{4}{c}{$\Q_2$--$\Q_1^\ast$}\\
Grid & Velocity dof & Pressure dof & Iters & $\gamma^2$ & Velocity dof & Pressure dof & Iters & $\gamma^2$\\
\hline
3 & 2187 & 125 & 45 & 0.1128 & 2187 & 189 & 48 & 0.1122\\
4 & 14739 & 729 & 51 & 0.1122 & 14739 & 1241 & 52 & 0.1115\\
5 & 107811 & 4913 & 51 & 0.1116 & 107811 & 9009 & 52 & 0.1110\\
\hline
\end{tabular}
\caption{Preconditioned MINRES iterations and discrete inf--sup constant approximations for Example \ref{ex:3d_cavity} 
and preconditioner $\mathcal{P}_1$.}
\label{tab:3d_infsup}
\end{table}

\begin{table}
\begin{tabular}{r |r r r r | r r r r}
\hline
& \multicolumn{4}{c|}{$\Q_2$--$\Q_1$} & \multicolumn{4}{c}{$\Q_2$--$\Q_1^\ast$}\\
Grid & Velocity dof & Pressure dof & Iters & $\gamma^2$ & Velocity dof & Pressure dof & Iters & $\gamma^2$\\
\hline
3 & 2187 & 125 & 50 & 0.1132 & 2187 & 189 & 53 & 0.1126\\
4 & 14739 & 729 & 54 & 0.1128 & 14739 & 1241 & 59 & 0.1124\\
5 & 107811 & 4913 & 56 & 0.1123 & 107811 & 9009 & 69 & 0.1094\\
\hline
\end{tabular}
\caption{Preconditioned MINRES iterations and discrete inf--sup constant approximations for Example \ref{ex:3d_cavity} 
and preconditioner $\mathcal{P}_2$.}
\label{tab:3d_infsup_sgscheb}
\end{table}

\begin{figure}
\begin{center}
\begin{subfigure}[b]{0.4\textwidth}
         \centering
         \includegraphics[width=\textwidth]{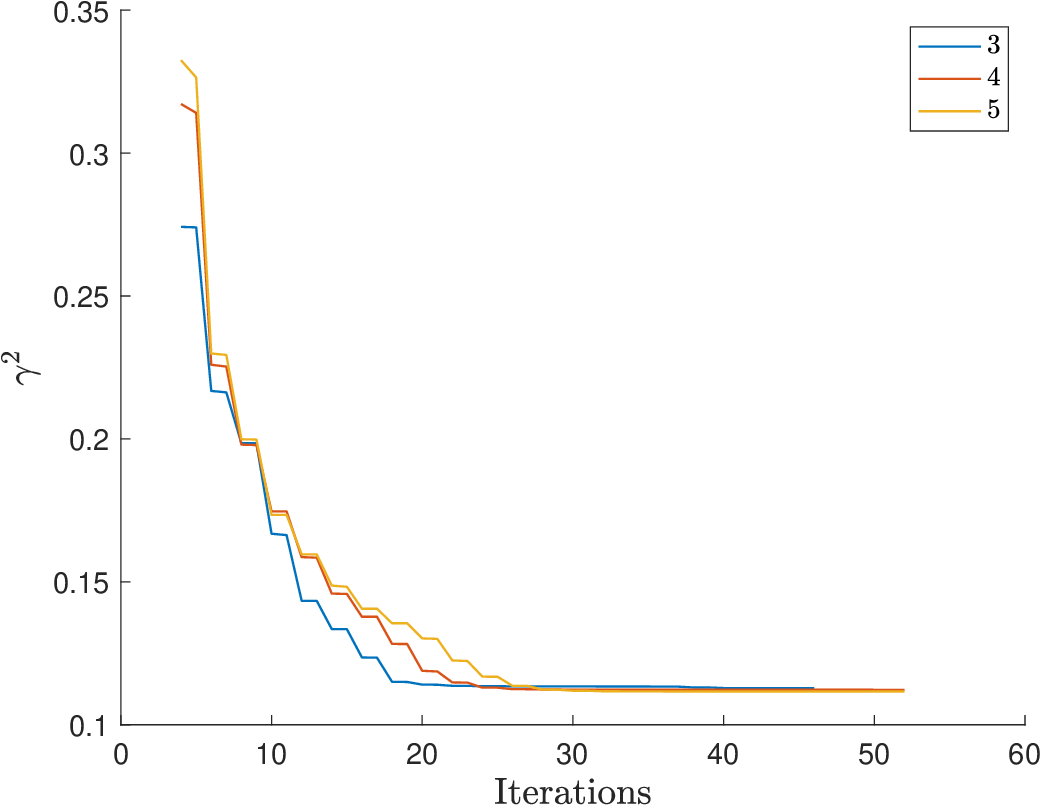}
     \end{subfigure}
     \begin{subfigure}[b]{0.4\textwidth}
         \centering
         \includegraphics[width=\textwidth]{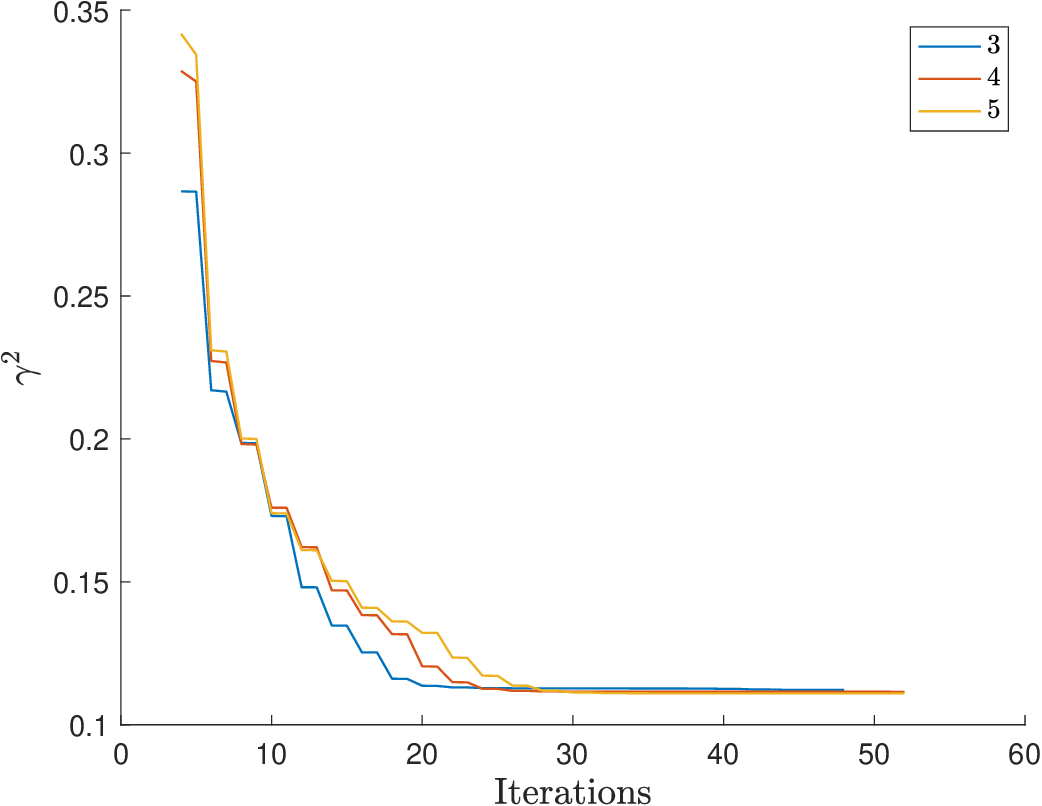}
     \end{subfigure}
     \end{center}
\caption{EST-MINRES estimates of the discrete inf--sup constant $\gamma^2$ at each iteration for Example \ref{ex:3d_cavity} 
and preconditioner $\mathcal{P}_1$
    with $\Q_2$--$\Q_1$ (left) and $\Q_2$--$\Q_1^\ast$ (right) elements for the grids specified in Table \ref{tab:3d_infsup}.}
\label{fig:3d_infsup}
\end{figure}

\section{Two-field pressure preconditioning strategies for Oseen flow.} \label{sec:oseen}

In this section we consider the Oseen problem
\begin{alignat*}{2}
- \nu \triangle \vfld{u} + \vfld{w} \cdot \nabla \vfld{u}+ \nabla p &= \vfld{0} & \quad & \text{ in } \Omega,\\
\nabla\cdot \vfld{u} & = 0& \quad & \text{ in } \Omega,\\
\vfld{u} &= \vfld{g} & \quad & \text{ on } \partial \Omega,
\end{alignat*}
that arises from applying fixed point iteration to the 
Navier--Stokes equations modelling  steady flow in a  channel domain
$\Omega \subset \R^2$ with viscosity parameter $\nu$.
The Oseen problem is a linear elliptic system of PDEs and assuming  
a divergence-free convection field $\vfld{w}$ and 
a smooth (e.g., polygonal) domain $\Omega$, 
has a unique weak solution for all positive values of the  viscosity parameter.\footnote{Uniquely 
defined  solutions to the underlying steady-state Navier-Stokes problem are only guaranteed when the 
viscosity parameter is sufficiently large; see~\cite[section\,8.2.1]{elman14}.  \bookref}  
The inf--sup stability of our two-field mixed approximation 
is a sufficient condition for convergence of the mixed  approximation  to the 
weak solution of the Oseen problem as $h\to 0$.

Our focus will be on the  associated
discrete matrix system, 
\begin{align} \label{oseen-system}
\left[
\begin{array}{@{}cc@{}}
\boldsymbol{F}  & B^T \\ B & 0
\end{array}
\right]
\left[
\begin{array}{@{}c@{}}
\bfu \\ \bfp
\end{array}
\right]
=
\left[
\begin{array}{@{}c@{}}
\bff \\ \bfg
\end{array}
\right],
\end{align}
where the unknown coefficient vector involves   
the discrete velocity vector $\bfu \in \R^{\nvelocity}$ and the
two-field pressure vector $\bfp \in \R^{n_p}$. 
The right-hand side vectors $\bff$ and $\bfg$ are  associated with the 
specified boundary velocity  (as in the Stokes flow case).
References to the ``residual norm''  throughout this section will
refer  to the standard Euclidean norm  for finite dimensional vectors. Thus, 
for the system \eqref{oseen-system} the convergence of our  iterative
solver strategies is assessed by monitoring 
\begin{align} \label{residual-norm}
\| \bfz \| :=
\left\| 
\left[
\begin{array}{@{}l@{}}
\bff - \boldsymbol{F}  \bfu   -  B^T \bfp \\
\bfg -B \bfu 
\end{array}
\right]
\right\|_{\ell_2} .
\end{align}

The nonsymmetry of the matrix $\boldsymbol{F}$  means that the
iterative solver of choice is GMRES (see~\cite[section\,9.1]{elman14}) \bookref 
together with a block preconditioning operator of the form
\begin{align} \label{nse-precon-diag}
\MM =
\left[
\begin{array}{@{}cc@{}}
\boldsymbol{M} & B^T \\ 0 & - M_S
\end{array}
\right],
\end{align}
where $\boldsymbol{M}$ is an optimal complexity (multigrid) operator effecting the action 
of the inverse of the matrix  $\boldsymbol{F}$ and 
$M_S$ is an optimal complexity approximation  of the Schur complement matrix
$B\boldsymbol{F}^{-1}B^T$.
We will discuss results for two representative flow problems herein. 

\begin{figure}[!ht]
       \begin{center}
	\includegraphics[width=0.95\linewidth]{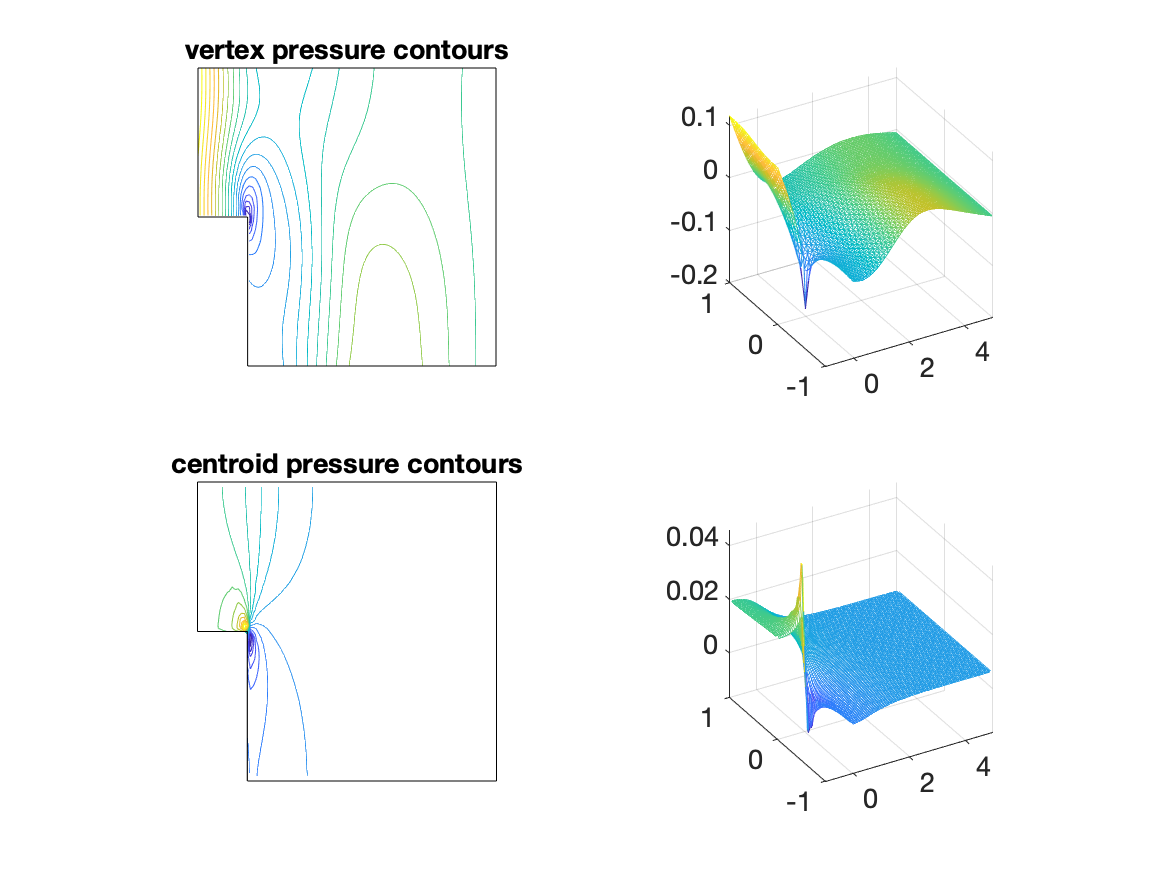}		 
	\end{center}
\caption{Representative $\P_2$--$\P_1^*$ pressure field solution for flow over a step 
($\RE = 100$) computed 
on a uniform   mesh with  5632 right-angled  triangles and 2945 vertices.}
\label{fig.testproblem1}
\end{figure}

\begin{examp}[flow over a step]\label{ex.step}
We consider an inflow--outflow problem
defined  on the domain  $D=[-1,0) \times [0,1]  \cup [0,5]\times[-1,1]$   with the viscosity parameter 
$\nu$ set to $1/50$ and  a parabolic velocity 
$u_x=  4y (1 - y)$ specified on the inflow boundary $x=-1$.  The  construction of  the standard
weak formulation (see \cite[p.\,127]{elman14}) \bookref 
gives rise to a  natural  boundary condition that fixes the hydrostatic pressure
level by weakly enforcing a zero mean pressure at the outflow boundary $x=5$. (The other 
boundary conditions are associated with fixed walls.)  
We consider  the discrete system \eqref{oseen-system}
that results after 5 fixed-point iterations  of the discretised Navier--Stokes system
starting from the corresponding Stokes flow solution.
We generate  solutions using  $\Q_2$--$\Q_1^\ast$  or $\P_2$--$\P_1^\ast$ augmented Taylor--Hood 
mixed approximation with the domain subdivided uniformly into (bisected) squares. 
The resulting system \eqref{oseen-system} is singular with
the one-dimensional pressure nullspace described in section~2. \backref
The two components of the pressure solution computed on a representative $\P_2$--$\P_1^\ast$ mesh 
are illustrated in Fig.~\ref{fig.testproblem1}.  The centroid pressure  values  are  an  order of magnitude
smaller than  the vertex pressure in all the elements---they provide the ``corrections'' to the
vertex pressures that are needed  to ensure  local (elementwise)  conservation of mass.
 \end{examp}

\begin{examp}[two-dimensional enclosed flow]\label{ex.cavity}
We consider the  classical driven-cavity enclosed flow problem
defined  on the domain  $D=[-1,1]^2$   with the viscosity parameter
$\nu$ set to $1/100$ and a nonzero tangential velocity
$u_y= 1- x^4$ specified on the top of the cavity. We take $\P_2$--$\P_1^\ast$
augmented Taylor--Hood mixed approximation with the domain subdivided uniformly into
$n^2$  bisected squares. We consider  the discrete system \eqref{oseen-system}
 that arises after 5 fixed-point iterations  starting from the corresponding Stokes flow solution.
 The discrete system is singular with a  two-dimensional pressure nullspace, corresponding
 to a constant vertex pressure and a constant centroid pressure.
The contour plot on the left in Fig.~\ref{fig.testproblem2} shows  the element divergence errors
 $\|\nabla \cdot \vfld{u}_h\|_{L^2(\triangle)}$  computed on a coarse mesh ($n=32$). The associated centroid
 ``correction'' pressure field contours shown on the plot  in the right can be seen to provide
 a good indication of the regions where the divergence error is concentrated.  As in
  Fig.~\ref{fig.testproblem1} the centroid  pressure  values  are  an  order of magnitude smaller
 than  the vertex pressure in all elements, so  the overall pressure field is visually identical to
 the $\P_1$ pressure field (not shown here).
 \end{examp} 
 \begin{figure}[!t]
       \begin{center}
	\includegraphics[width=0.7\linewidth]{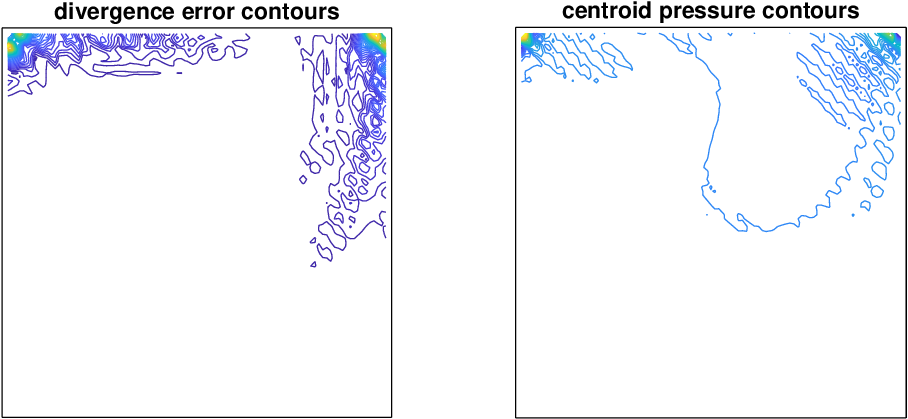}		 
	\end{center}
\caption{Comparison of divergence error and $\P_2$--$\P_1^*$  centroid pressure solution for
cavity flow  ($\RE=200$) computed on a uniform  mesh with 2048 right-angled  triangles and 1089 vertices.}
\label{fig.testproblem2}
\end{figure}

\subsection{Pressure convection-diffusion preconditioning for Oseen flow.} \label{sec:pcd}
There are two alternative ways of approximating the key matrix
$B\boldsymbol{F}^{-1}B^T$ in the case that $B^T$ is generated by a two-field
pressure approximation so that $B^T = [ B_1^T, B_0^T ]$.
 The focus will be on {pressure convection-diffusion} (PCD)
 preconditioning  in this section. 
Both  of the  Schur complement  approximations can be motivated by starting with the Oseen 
matrix operator (\ref{oseen-system}) and observing that the 
diagonal blocks of $\boldsymbol{F}$
are discrete representations of the convection--diffusion operator
\begin{align} \label{generic-conv-diff}
\mathcal{L} = -\nu \nabla^2 + \vfld{w}_h \cdot \nabla ,
\end{align}
defined on the velocity space. In practical calculations $\nu>0$ is proportional to the 
inverse of the flow Reynolds number  and
$\vfld{w}_h$ is the discrete approximation to the flow 
velocity computed at the most recent nonlinear iteration.
The PCD approximation  supposes that there is an analogous operator
to  (\ref{generic-conv-diff}), namely
\begin{align}\label{pressure-conv-diff}
\mathcal{L}_p = (-\nu \nabla^2 + \vfld{w}_h \cdot \nabla)_p
\end{align}
defined on the two components of the augmented  pressure space. 

To this end,  defining $\{\phi_j\}_{j=1}^{n_{1}}$   to be  the basis for the 
$C^0$ pressure discretisation, we construct matrices $Q_1$ and $F_1$ so that
\begin{align*}
Q_1 &=[q_{1,ij}],{\quad}q_{1,ij} = \int_{\Omega}
\phi_j \,  \phi_i \\
F_1 &= [f_{1,ij}], \quad f_{1,ij} = 
\nu \int_{\Omega} \nabla \phi_j \cdot \nabla \phi_i +
\int_{\Omega} (\vfld{w}_h \cdot \nabla \phi_j) \, \phi_i .
\end{align*}
We then note that if the commutator  with the divergence operator
\begin{align*} 
\mathcal{E} = \nabla \cdot (-\nu \nabla^2 + \vfld{w}_h \cdot \nabla) -
(-\nu \nabla^2 + \vfld{w}_h \cdot \nabla)_p \, \nabla \cdot
\end{align*}
is  small then we have the approximation
\begin{align} \label{discrete-commutator1}
0 \approx (Q_1^{-1}B_1)\, (\boldsymbol{M}^{-1}\boldsymbol{F}) -
(Q_1^{-1}F_1) \, (Q_1^{-1}B_1) 
\end{align}
where $\boldsymbol{M}$ is the diagonal of the mass matrix associated with
the basis representation of the velocity space.\footnote{The inverse of the mass matrix is
a dense matrix. The diagonal of   the mass matrix is a spectrally equivalent matrix operator
with a sparse (diagonal)  inverse.} Rearranging  
 (\ref{discrete-commutator1}) gives the first Schur complement approximation
 \begin{align} \label{discrete-commutator1x}
B_1 \, \boldsymbol{F}^{-1} B^T \approx  Q_1 F_1^{-1} (B_1\, \boldsymbol{M}^{-1} B^T) .
\end{align}

A discrete version of  $\mathcal{L}_p$  for
the piecewise constant pressure space can be generated by considering the
jumps in pressure across inter-element boundaries; see 
\cite[pp.\,268--370]{elman14}. \bookref 
To this end, defining $\{\varphi_j\}_{j=1}^{n_{0}}$   to be  the (indicator function)  basis for the 
discontinuous  pressure, we construct matrices $Q_0$ and $F_0$ via
\begin{align*}
Q_0 &=[q_{0,ij}],{\quad}q_{0,ij} = \int_{\Omega}
\varphi_j \,  \varphi_i  =  
\begin{cases} |T_i | & \hbox{if }  i=j , \\
                      \; 0 & \hbox{otherwise},  \end{cases} \\
F_0 &= [f_{0,ij}], \quad f_{0,ij} = 
\nu \sum_{T\in \TT_h} \int_{T} \nabla \varphi_j \cdot \nabla \varphi_i +
\sum_{T\in \TT_h}  \int_{T} (\vfld{w}_h \cdot \nabla \varphi_j) \, \varphi_i 
\end{align*}
and note that if the commutator  with the divergence operator
is  small then we have 
\begin{align} \label{discrete-commutator2}
0 \approx (Q_0^{-1}B_0)\, (\boldsymbol{M}^{-1}\boldsymbol{F}) -
(Q_0^{-1}F_0) \, (Q_0^{-1}B_0) ,
\end{align}
suggesting the second Schur complement approximation
 \begin{align} \label{discrete-commutator2x}
B_0\, \boldsymbol{F}^{-1} B^T \approx  Q_0 F_0^{-1} (B_0 \, \boldsymbol{M}^{-1} B^T) .
\end{align}

Combining \eqref{discrete-commutator1x} with \eqref{discrete-commutator2x} then gives
a two-field PCD approximation
\begin{align} \label{twofield-pcd}
B\boldsymbol{F}^{-1}B^T \approx M_S :=
\left[
\begin{array}{@{}cc@{}}
Q_1 & 0 \\  0 & Q_0 \end{array}
\right]
\left[
\begin{array}{@{}cc@{}}
F_1^{-1} & 0 \\  0 & F_0^{-1}  \end{array}
\right]
 B  \boldsymbol{M}^{-1} B^T .
\end{align}

Two features of the PCD approximation \eqref{twofield-pcd} are worth noting.
The first point is that the coupling  between the pressure components is 
represented by  the $2\times 2$ block matrix  $B \boldsymbol{M}^{-1} B^T$ rather than
by the pressure mass matrix $\pressuremass$. 
The second key point is that the matrices $M_S$ and 
$B  \boldsymbol{F}^{-1} B^T$
have the same nullspace, independent of the nature of underlying flow problem that is 
being solved. 

 \begin{figure}[!th]
       \begin{center}	
	\includegraphics[width=0.35\linewidth]{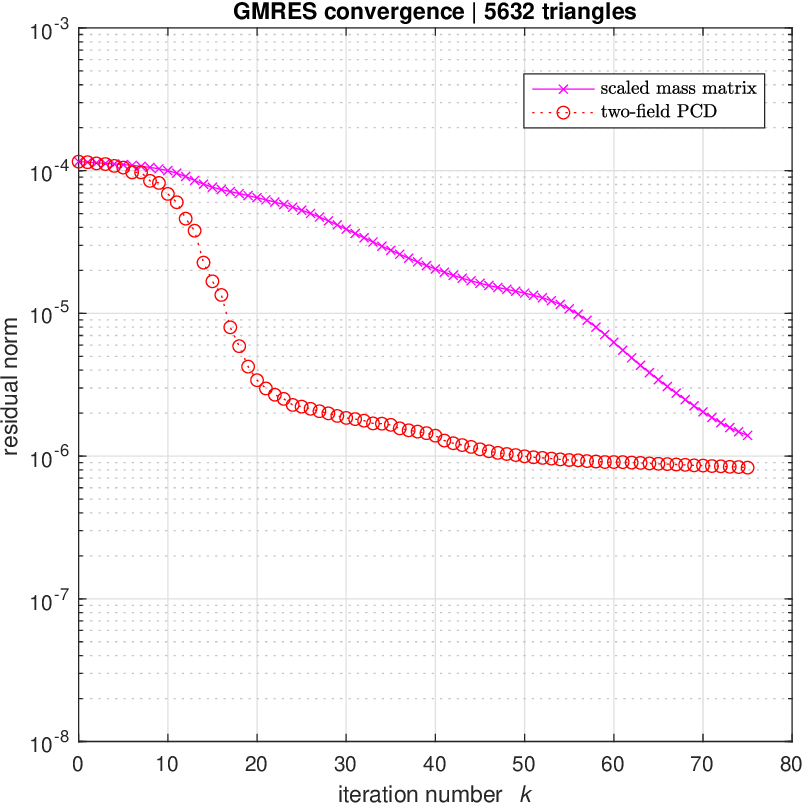}		
	\includegraphics[width=0.35\linewidth]{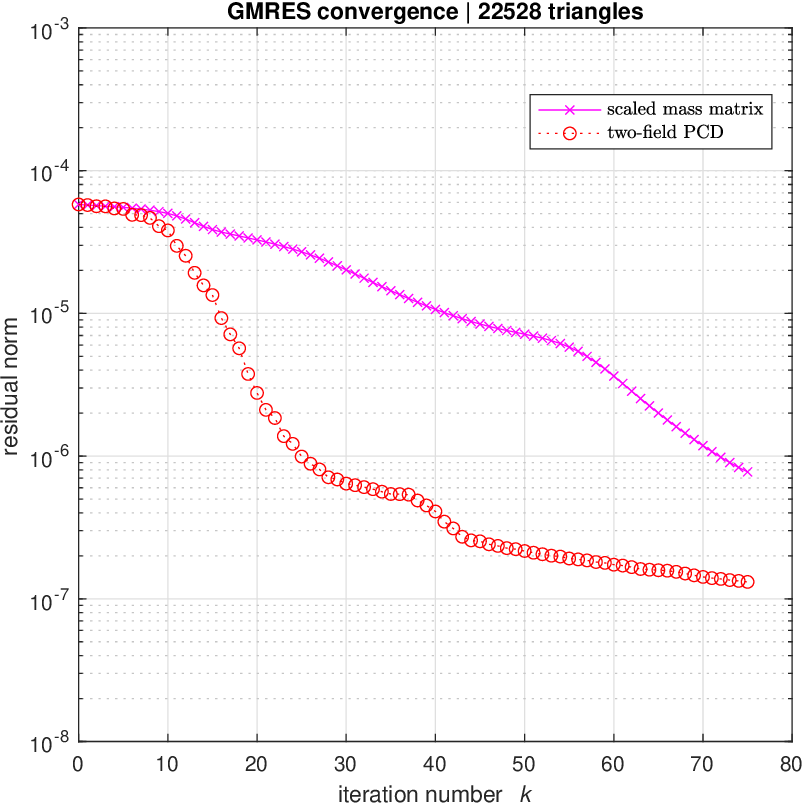}	 
	\end{center}
\caption{Absolute residual reduction  for test problem~3  
when computing $\P_2$--$\P_1^*$   solutions
using preconditioners  $\MM_1$ ({\textcolor{magenta}x}) 
or $\MM_2$ ({\textcolor{red}o})  on  two nested meshes.}
\label{fig.testproblem1_p2p1}
\end{figure}

The PCD approximation  in \eqref{twofield-pcd}  is  imperfect in practice.
To illustrate this, representative convergence histories  that arise in solving the inflow-outflow problem
using $\P_2$--$\P_1^*$  approximation (shown in Fig.~\ref{fig.testproblem1})
 are presented in Fig.~\ref{fig.testproblem1_p2p1}.  Taking the solution from the
 previous Picard iteration as the initial guess, we note  that the initial residual 
 norm of the target matrix system \eqref{oseen-system} is close to $10^{-4}$
 independent of the spatial discretisation.
 Convergence plots are presented for two preconditioning strategies, namely
\begin{align} \label{nse-precon-ref}
\MM_1 =
\left[
\begin{array}{@{}cc@{}}
\boldsymbol{F} & B^T \\ 0 & - {1\over \nu} \pressuremass
\end{array}
\right], \quad
\MM_2 =
\left[
\begin{array}{@{}cc@{}}
\boldsymbol{F} & B^T \\ 0 & - M_S
\end{array}
\right],
\end{align}
with $M_S$ defined in \eqref{twofield-pcd}. 
The first strategy is the  block triangular extension of  the 
Stokes preconditioning strategy discussed in section~2. \backref
We note that the resulting convergence is very slow---around 50 iterations are 
required to reduce the residual norm by an order of magnitude---but is independent of the 
discretisation level.  In contrast we see that the PCD preconditioning  strategy 
has two distinctive phases of convergence behaviour. An initial phase of relatively 
fast convergence is followed by a secondary phase where GMRES 
stagnates. We  hypothesise that this stagnation is a consequence of the 
ill-conditioning of the  eigenvectors of the matrix operator $M_S$. A
notable feature is that the onset of the stagnation is delayed when
solving the same problem on  a finer grid. The two-phase convergence 
behaviour is ubiquitous---the same pattern  is seen  using 
rectangular elements and the stagnation  does not go away when 
the viscosity parameter is increased from $1/50$ to $1/5$. An alternative strategy is
clearly needed!

One way of designing a more robust PCD preconditioning strategy for a two-field pressure 
approximation   is  to exploit the fast convergence of the PCD
approximation for the unaugmented  Taylor-Hood approximation.
The starting point for  such a strategy  is to rewrite the system \eqref{oseen-system} in the form
\begin{align} \label{full-system}
\left[
\begin{array}{@{}ccc@{}}
\boldsymbol{F}  & B_1^T & B_0^T\\ B_1 & 0 & 0 \\ B_0 & 0 & 0
\end{array}
\right]
\left[
\begin{array}{@{}l@{}}
\bfu \\ \bfp_1 \\  \bfp_0
\end{array}
\right]
=
\left[
\begin{array}{@{}l@{}}
\bff \\  \bf{0} \\ \bf{0}
\end{array}
\right] .
\end{align}
The proposed solution algorithm is then a two-stage process.

\begin{itemize} 
\item{\bfseries Input:}  residual reduction tolerance $\eta$ and reduced system residual factor $c$
\item[Step I] 

Generate  a PCD  solution  to  the {\it reduced system} 
\begin{align} \label{reduced-system}
\left[
\begin{array}{@{}cc@{}}
\boldsymbol{F}  & B_1^T \\ B_1 & 0
\end{array}
\right]
\left[
\begin{array}{@{}c@{}}
\bfu_1 \\ \bfq_1
\end{array}
\right]
=
\left[
\begin{array}{@{}c@{}}
\bff \\  \bf{0}
\end{array}
\right]
\end{align}
using  the Schur complement approximation \eqref{discrete-commutator1}, stopping
the GMRES iteration when the residual is reduced by a factor of $c\eta$.

\item[Step II] 
Generate  a  solution  to the target system 
\eqref{full-system} with residual tolerance  $\eta$ using 
preconditioning strategy  $\MM_1$ in \eqref{nse-precon-ref} with the refined  initial guess 
$\lbrack \bfu_1^*  , \bfq_1^*, {\bf 0} \rbrack $.

\item
{\bfseries Output:}  refined solution $\lbrack \bfu^*  , \bfp_1^*,  \bfp_0^* \rbrack $
\end{itemize}

\begin{figure}[!ht]
       \begin{center}
	\includegraphics[width=0.35\linewidth]{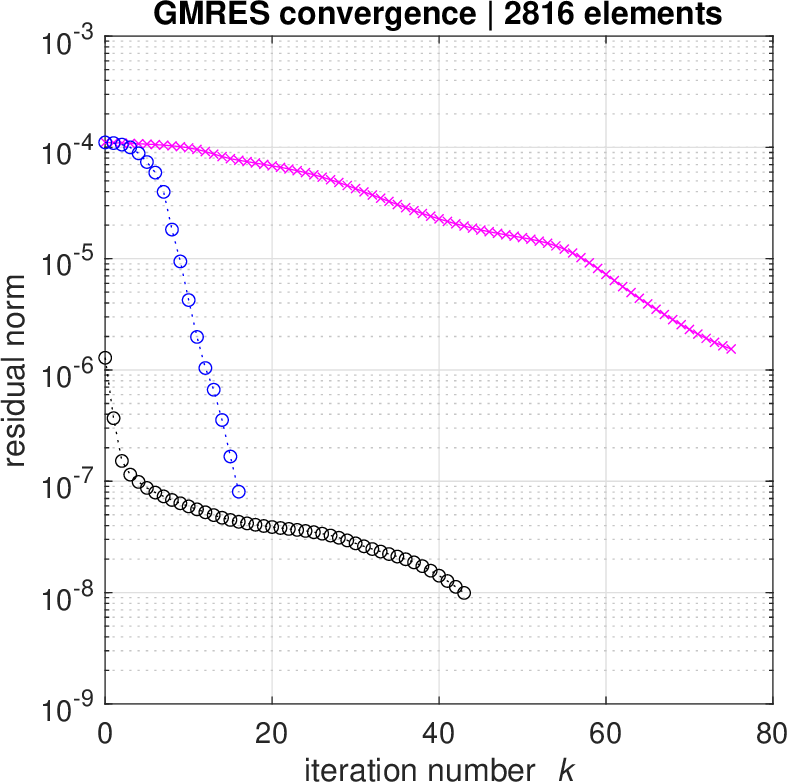}		
	\includegraphics[width=0.35\linewidth]{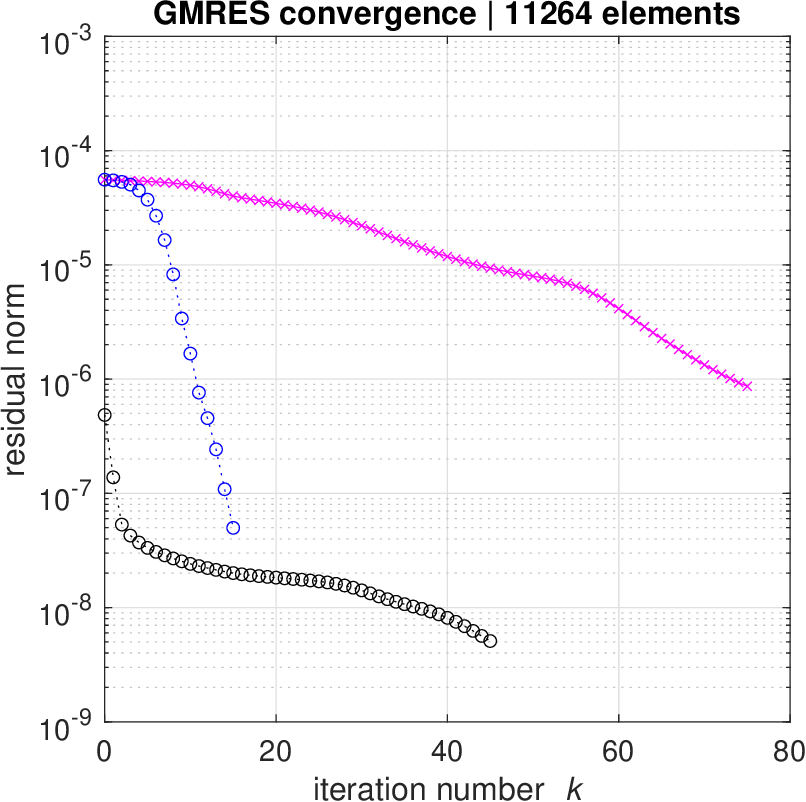}	 
	\end{center}
\caption{Absolute residual reduction  for test problem~3  
when computing  $\Q_2$--$\Q_1^*$   solutions using 
using preconditioners  $\MM_1$ ({\textcolor{magenta}x}) 
or two-stage PCD  (Step I: {\textcolor{blue}o}; Step II: {\textcolor{black}{o}})
  on  two nested meshes.}
\label{fig.testproblem1_q2q1}
\end{figure}

\begin{figure}[!ht]
       \begin{center}
	\includegraphics[width=0.35\linewidth]{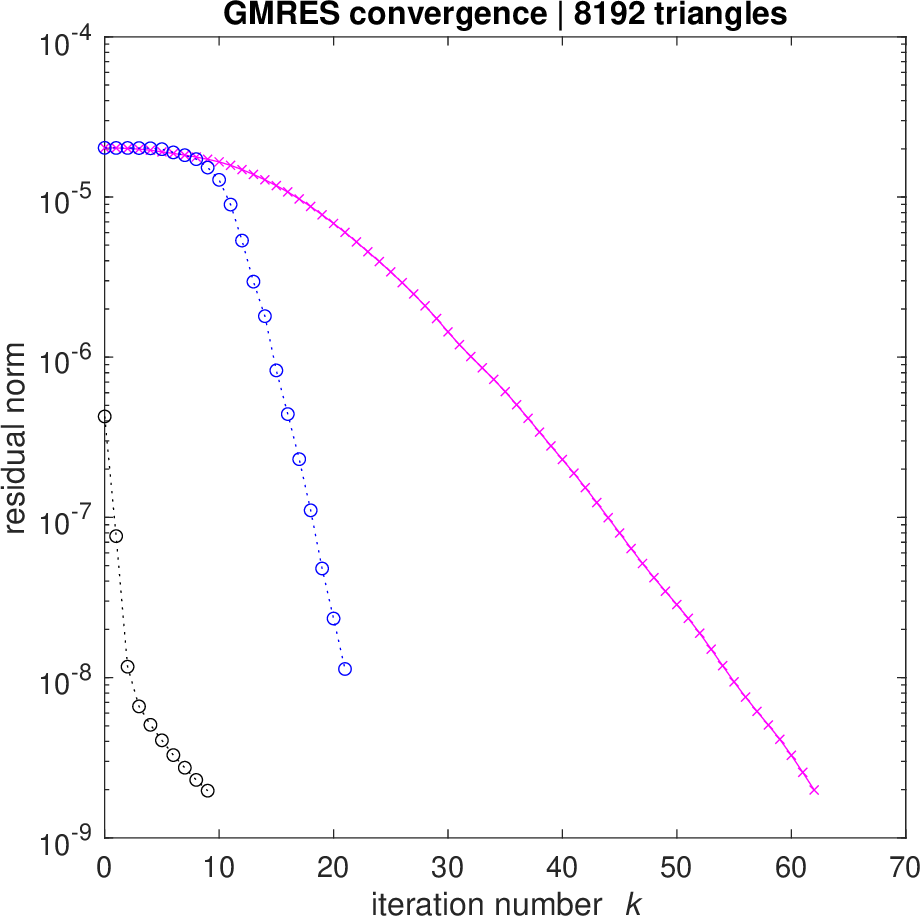}		
	\includegraphics[width=0.35\linewidth]{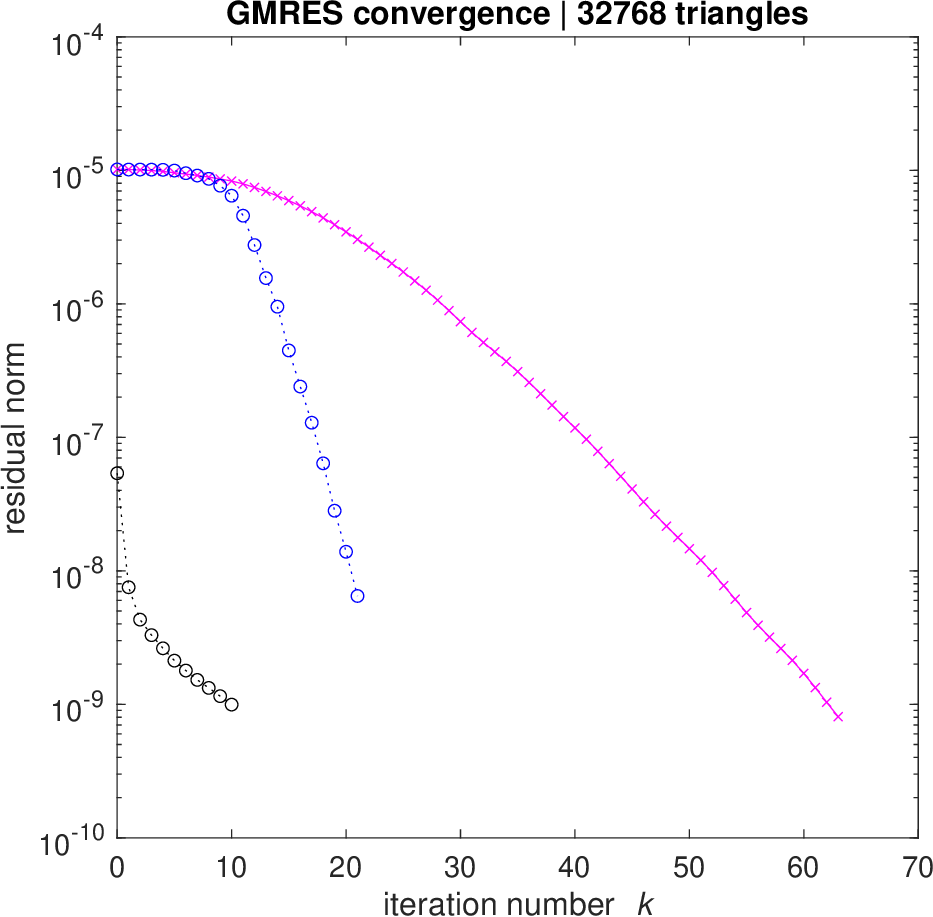}	 
	\end{center}
\caption{Absolute residual reduction  for test problem~4  
when computing  $\P_2$--$\P_1^*$   solutions using 
using preconditioners  $\MM_1$ ({\textcolor{magenta}x}) 
or two-stage PCD  (Step I: {\textcolor{blue}o}; Step II: {\textcolor{black}{o}})
 on  two nested meshes.}
\label{fig.testproblem2_p2p1}
\end{figure}

Sample  results  generated using this strategy with  tolerance 
$\eta$ set to $10^{-4}$ and $c$ set to 10 are presented  in Fig.~\ref{fig.testproblem1_q2q1}.   
Sample results for the cavity test problem
are presented in Fig.~\ref{fig.testproblem2_p2p1}.   
The results in Fig.~\ref{fig.testproblem1_q2q1} and Fig.~\ref{fig.testproblem2_p2p1}
are representative of the two-stage convergence profiles that are generated when solving 
these test problems at other Reynolds numbers.  Our experience is that the level of
residual reduction  is perfectly robust with regards to
the spatial discretisation---typically giving  smaller iteration counts when the mesh resolution 
is increased (a known feature of PCD preconditioning). The convergence rates of both
stages of the algorithm deteriorate slowly when the Reynolds number is increased. 
Our strategy for terminating the first stage of the iteration is  motivated by the following result.
\begin{proposition}\label{pr.residualbound}
The residual error  $\| \bfz^* \|$
associated with the intermediate  solution $\lbrack \bfu^*_1  , \bfq_1^*,  \bf{0} \rbrack $
to the discrete system \eqref{full-system} satisfies the bound
\begin{align} \label{resbound}
\| \bfz^* \|^2  &\leq c^2 \eta^2 \,  \| \bff \|^2  + \| B_0 \bfu_1^* \|^2  ,
\end{align}
where  $\| \bff \|$ is the initial residual error associated with  a zero initial vector.

\begin{proof}
The vector  $ \bfz^*$  associated with the intermediate solution is the 
three-component vector
\begin{align} \label{res-system}
\left[
\begin{array}{@{}l@{}}
\bfr^* \\ \bfr_1^* \\ \bfr_0^*
\end{array}
\right]
=
\left[
\begin{array}{@{}l@{}}
\bff \\  \bf{0} \\ \bf{0}
\end{array}
\right] 
-
\left[
\begin{array}{@{}ccc@{}}
\boldsymbol{F}  & B_1^T & B_0^T\\ B_1 & 0 & 0 \\ B_0 & 0 & 0
\end{array}
\right]
\left[
\begin{array}{@{}l@{}}
\bfu_1^* \\ \bfq_1^* \\  \bf{0}
\end{array}
\right] .
\end{align}
The stopping test for solving the reduced system ensures that 
\begin{align}
 \| \bfr^* \|^2   +  \| \bfr_1^* \|^2  \leq c^2 \eta^2 \, \| \bff \|^2.
 \end{align}
 Thus we have 
 \begin{align}
\| \bfz^* \|^2 =   \| \bfr^* \|^2   +  \| \bfr_1^* \|^2  
  +   \| \bfr_0^* \|^2  \leq c^2 \eta^2  \, \| \bff \|^2 +  \| B_0 \bfu_1^* \|^2.  
 \end{align}
 \end{proof}
\end{proposition}
The bound  \eqref{resbound}  has two terms on the right-hand side. While the first term 
can be controlled by reducing $\eta$ and/or $c$,  the second term  measures the  
local incompressibility of the intermediate  Taylor--Hood solution
$ \| B_0 \bfu_1^* \|^2 =  \sum_j \big( \int_{T_j} \! \nabla \cdot {\vec{u}_h}^{\,*} \big)^2$,
where ${\vec{u}_h}^{\,*}$ is the expansion of the coefficient vector $\bfu_1^* $ in the basis of the 
velocity approximation space. Setting  $\eta = 10^{-4}$ and $c=10$ we see that the second term saturates 
the residual error and the residual error jumps up  when the switch is made from the first to the  
second step of the algorithm.  This ``transition'' jump in the residual norm is
clearly evident in the convergence plots. The convergence in the second step
is rapid initially but eventually mirrors the rate observed for  $\MM_1$ approximation
with a standard starting guess.

\subsection{Least square commutator approximation for Oseen flow} \label{sec:lsc}
A second way of approximating the key matrix
$B\boldsymbol{F}^{-1}B^T$ in the case that $B^T$ is generated by a two-field
pressure approximation is  given by the  {least-squares commutator} (LSC) preconditioner
 \begin{align} \label{lsc-definition}
B \boldsymbol{F}^{-1} B^T \approx   M_S =
(B  \boldsymbol{H}^{-1} B^T )\, 
(B \boldsymbol{M}^{-1}  \boldsymbol{F} \boldsymbol{H}^{-1} B^T )^{-1}  
\, (B  \boldsymbol{M}^{-1} B^T) .
\end{align}

The attractive feature of LSC  is that the construction of $M_S$ is 
completely algebraic. The only technical issue is the need to make adjustments   
 on rows and columns associated with tangential velocity degree of freedom adjacent to
  inflow and fixed wall boundaries. These adjustments are associated  with
 a diagonal   scaling matrix $\boldsymbol{D}$ so that
$\boldsymbol{H}= \boldsymbol{D}^{-1/2} \boldsymbol{M} \boldsymbol{D}^{-1/2}$.
Full details can be found in \cite[pp.\,376--379]{elman14}). \bookref

The LSC approximation \eqref{lsc-definition} is also far from perfect  when 
using triangular elements.\footnote{The strategy is designed for
tensor-product approximation spaces. It gives good results in the case of $\Q_2$--$\Q_1^\ast$.} 
To illustrate this, representative convergence histories  that arise in 
solving the inflow-outflow problem using $\P_2$--$\P_1^*$  approximation
 are presented in Fig.~\ref{fig.testproblem1_lsc}.   
 Convergence plots are presented for two preconditioning strategies, namely
 the refined PCD from the previous section  and 
\begin{align} \label{nse-precon-lsc}
\MM_3 =
\left[
\begin{array}{@{}cc@{}}
\boldsymbol{F} & B^T \\ 0 & - M_S
\end{array}
\right],
\end{align}
with $M_S$ defined in \eqref{lsc-definition}. 

\begin{figure}[!ht]
       \begin{center}
	\includegraphics[width=0.35\linewidth]{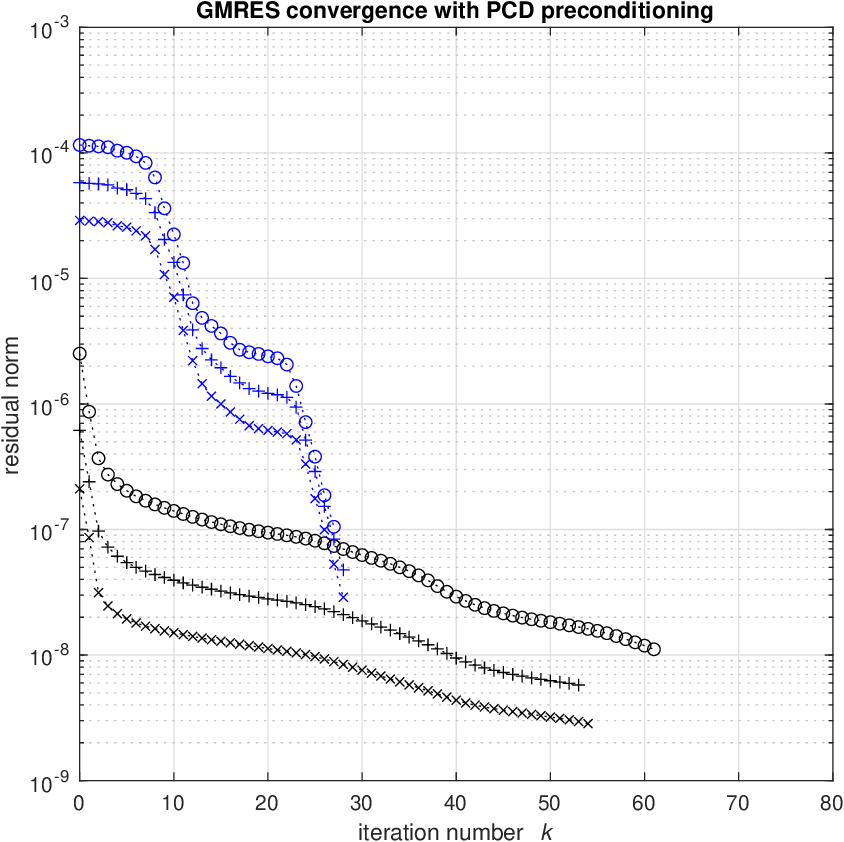}		
	\includegraphics[width=0.35\linewidth]{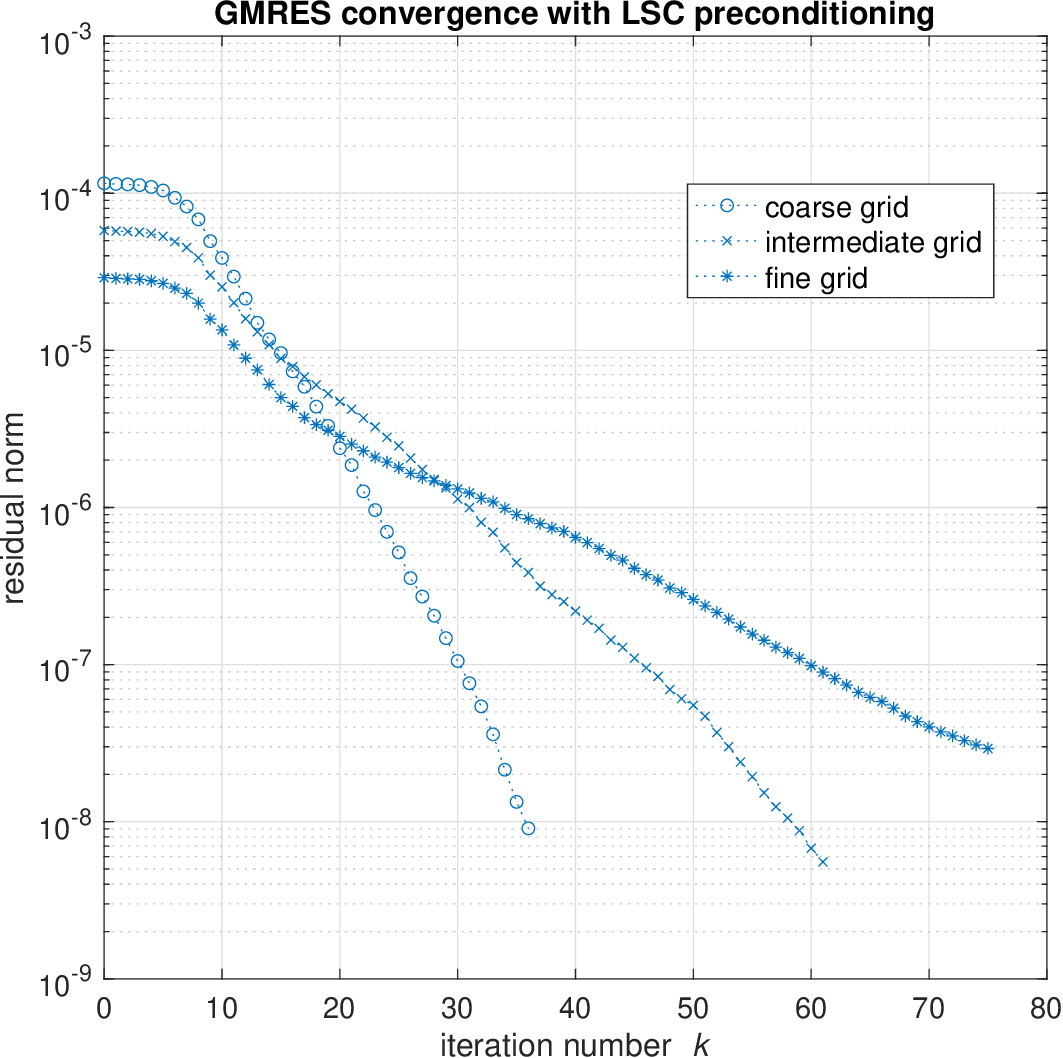}	 
	\end{center}
\caption{Absolute residual reduction  for test problem~3 
when computing  $\P_2$--$\P_1^*$   solutions using 
using  refined PCD  (left) or preconditioner  $\MM_3$ (right)
on  three nested meshes.}
\label{fig.testproblem1_lsc}
\end{figure}

The PCD  histories in Fig.~\ref{fig.testproblem1_lsc}
display  mesh independent convergence and are comparable with those 
observed using square elements  in Fig.~\ref{fig.testproblem1_q2q1}.
The associated cpu times for the solution on the intermediate grid  
with  $2\times 11264$ elements  were 13 seconds for the first step (28 iterations) and
26 seconds for the second step (53 iterations). The LSC preconditioning
strategy is not robust---the convergence rate deteriorates with increasing 
grid refinement. The cpu time for generating  the LSC solution on the
intermediate grid  was over 100 seconds  (61 iterations).

\section{Conclusions} \label{sec:conclusions}
Two-level pressure approximation for incompressible flow problems offer the
prospect of accurate approximation with minimal computational overhead.
Derived quantities of practical importance such as the mean wall shear stress 
are likely to be  computed much more precisely if incompressibility is enforced 
locally.\footnote{See
 \url{https://personalpages.manchester.ac.uk/staff/david.silvester/lecture1.18.mp4}
for a comparison of alternative strategies for computing the average shear stress   
for flow over a step at Reynolds number 200.} 
However, the augmented pressure space causes some challenges for the linear algebra
when the pressure space is defined by combining the usual Taylor--Hood pressure 
basis functions with basis functions for the piecewise continuous pressure space. This is 
 because constant functions can be expressed using either the usual (continuous) 
Taylor--Hood pressure space, or the augmented piecewise constant pressure space.
Specifically, with this common choice for specifying the pressure space, 
the pressure mass matrix becomes singular, and care should be taken to 
carefully construct and apply preconditioners that involve this matrix. 
Care should also be exercised when approximating the discrete inf--sup constant,  
and we find that naive approaches are not always reliable. 
On the other hand, the approximation implemented in EST-MINRES is robust.  
Our computational experimentation indicates that our two-stage  PCD strategy could be
the best way of iteratively solving  two-level pressure discrete linear algebra systems 
in the sense of  algorithmic reliability  and computational efficiency.

\bibliographystyle{siam}
\bibliography{twolevel}

\end{document}